\documentclass{article}
\usepackage[utf8]{inputenc}
\usepackage[T1]{fontenc}
\usepackage[ngerman,english]{babel}
\usepackage[a4paper, left=2.5cm, right=2.5cm, top=2.5cm]{geometry}

\usepackage{amsmath}
\usepackage{amssymb}
\usepackage{amsthm}
\usepackage{graphicx}
\usepackage{epstopdf}
\usepackage{color}
\usepackage[titletoc]{appendix}
\usepackage{url}
\usepackage{longtable}
\usepackage{setspace}
\usepackage{currvita}
\usepackage{floatflt}
\usepackage{authblk}

\usepackage[normalem]{ulem}
\pdfmapfile{=zxxbw7z.map}
\AtBeginDocument{%
\setlength{\cvlabelsep}{1em}%
\settowidth{\cvlabelwidth}{%
\cvlabelfont september/8888 -- september/8888%
}%
}%

\usepackage[cal=boondox]{mathalfa}

\newenvironment{rcases}{%
  \left.%
  \begin{cases}}%
{\end{cases}\right\rbrace}

\newtheorem{defi}{Definition}[section]
\newtheorem{lemma}[defi]{Lemma}
\newtheorem{theorem}[defi]{Theorem}
\newtheorem{remark}[defi]{Remark}
\newtheorem{fol}[defi]{Corollary}

\newtheorem{prob}[defi]{Problem}

\title{A Unified Approach to Scalar, Vector, and Tensor Slepian Functions on the Sphere and Their Construction by a Commuting Operator}
\author[$\dagger$]{V.~Michel}
\author[$\star$]{A.~Plattner}
\author[$\dagger$]{K.~Seibert}
\affil[$\dagger$]{Geomathematics Group, Department of Mathematics, University of Siegen, Germany}
\affil[$\star$]{Department of Geological Sciences, The University of Alabama, Tuscaloosa, AL, USA}

\begin{document}

\maketitle

\tableofcontents
\vspace*{2cm}

\newpage

\noindent\textbf{Abstract}
\newline
\newline
We present a unified approach for constructing Slepian functions
--- also known as prolate spheroidal wave functions --- on the sphere
for arbitrary tensor ranks including scalar, vectorial, and rank 2
tensorial Slepian functions, using spin-weighted spherical harmonics.
For the special case of spherical cap regions, we derived commuting
operators, allowing for a numerically stable and computationally
efficient construction of the spin-weighted spherical-harmonic-based
Slepian functions. Linear relationships between the spin-weighted and
the classical scalar, vectorial, tensorial, and higher-rank spherical
harmonics allow the construction of classical spherical-harmonic-based
Slepian functions from their spin-weighted counterparts, effectively
rendering the construction of spherical-cap Slepian functions for any
tensorial rank a computationally fast and numerically stable task.

\section{Introduction}

Functions cannot simultaneously have both a spectral and temporal (or spatial) finite support \cite{freedenmichelsimons,NarcWard1996}. In scientific or engineering applications, however, it may be desirable to represent signals in a time-limited but spectrally concentrated manner. Slepian, Landau, and Pollak created a suitable orthogonal basis for a range of Euclidean domains \cite{landaupollak, slepianpaper, slepianpollak}, see also \cite{grunbaum, simons}. Geoscientific or planetary studies typically involve data or models on a sphere, or parts thereof. To reap the benefits of the spatiospectral analysis previously developed for Euclidean spaces, Albertella, Sans\`o, and Sneeuw \cite{sneeuw}, and later Simons, Dahlen, and Wieczorek \cite{simons,simonspaper} developed corresponding scalar-valued functions by spatiospectrally optimizing linear combinations of spherical harmonics and named the resulting orthogonal basis ``Slepian functions''.
These Slepian functions found a wide range of applications in fields such as geodesy and geophysics, gravimetry, geodynamics, cosmology, planetary science, biomedical science, and in computer science (see \cite{jahn, plattner} and the references therein).
\newline
\newline
As a result of their construction, Slepian functions can be orthonormalized on the sphere, while remaining orthogonal within the target region. The two end member families of Slepian functions include spatially concentrated -- spectrally limited bases, and spatially limited -- spectrally concentrated bases. Here, we limit our discussions on the former case.
\newline
\newline
Construction of the spatially concentrated -- spectrally limited Slepian functions requires solving a finite-dimensional algebraic eigenvalue problem. Depending on the region of interest and the bandlimit, the underlying matrix can become ill-conditioned, leading to an eigenvalue problem, which is numerically unstable to solve. For some special regions, alternative eigenvalue problems based on commuting operators were discovered for the scalar and the vectorial case on the sphere (see \cite{jahn, simons, simonspaper}). The alternative problems have the same eigenvectors but are numerically stable. 
\newline
\newline
The construction of Slepian functions for various tensorial ranks typically follows the recipe described in \cite{michelsimons}, where an abstract Hilbert space setup is used to describe the general ``construction manual'' of Slepian functions with a particular focus on ill-posed inverse problems but without commuting operators. In this article, we restrict our considerations on the case of a spherical cap region. We show that the salar, vectorial, and tensorial cases can be considered as particular cases of a generalized setup, which we present here.
\newline
\newline
For this purpose, we utilize the spin-weighted spherical harmonics of Newman and Penrose \cite{newmanpenrose} (see also \cite{handbook, diss}). As a consequence the eigenvalue problems for vectorial or higher-ranked Slepian function constructions, which are coupled when using the classical vector or tensor spherical harmonics, decouple. Besides reducing the dimension of the eigenproblem, this decoupling also facilitates the construction of a general commuting
operator for polar cap regions. The special case of spin weight $0$ leads to the known scalar Slepian functions (see also \cite{simons,simonspaper}), while the combination of spin weight $0$ with spin weight $\pm 1$ yields the known vector Slepian functions (see also \cite{jahn, plattner}).
\newline
\newline
 We demonstrate our method by explicitly constructing tensor (rank 2) Slepian functions using spin weights of $0$, $\pm1$, and $\pm2$. We compare the result to a system of Slepian functions constructed from the basis of the tensor spherical harmonics of Freeden, Gervens, and Schreiner \cite{freedenpaper}, by transforming the latter into the spin-weighted basis system. Tensor Slepian functions have been constructed before using a different ansatz by \cite{eshagh} for the basis of the tensor spherical harmonics of Martinec \cite{martinec}. Our general ansatz yields, for the first time, a commuting operator for the tensorial
case and it opens a way for the consideration of tensors of arbitrary ranks. This paper comprises some of the results of the thesis \cite{diss}.

\section{Preliminaries}

Before presenting the spin-weighted spherical harmonics by Newman and Penrose \cite{newmanpenrose}, upon which our unified Slepian construction is based, we recall the classical scalar, vector \cite{hill,morse}, and tensor spherical harmonics \cite{freedenpaper}. This will allow us to compare the two constructions in Theorem~\ref{relation_tens23}. And we describe a procedure to construct classical Slepian functions from spin-weighted Slepian functions, should the need arise. As usual, $\mathbb{N}$, $\mathbb{Q}$, $\mathbb{R}$, and $\mathbb{C}$ stand for the sets of positive integers, rational numbers, real numbers, and complex numbers, respectively. Correspondingly, $\mathbb{Q}_0^+ := \{x\in\mathbb{Q}\,|\,x\geq 0\}$ etc.

\subsection{Basics and Notations}

For completeness, we present the definitions for function spaces, norms, unit vectors and tensors, and operators that we use. For further details, see \cite{freeden, volker}.

\begin{defi}
Let $D\subset \mathbb R^n$ and $W\subset \mathbb{C}^m $. Then, $\mathrm C^{(k)}(D,W)$ is the space of all functions $F:D\to W$, which are at least differentiable to order $k \in \mathbb N_0$ and where the $k$-th derivative is continuous. If $W= \mathbb C $, then we denote $\mathrm C^{(k)}(D, \mathbb C ) =: \mathrm C^{(k)}(D)$ and if $k=0$, then we write $\mathrm C^{(0)}(D, W ) =: \mathrm C(D,W)$.
\newline
\newline
Analogously, we define $\mathrm c^{(k)}(D,w)$ for all vector functions $f:D\to w$, $w\subset \mathbb{C}^3 $, and $\boldsymbol{\mathrm c}^{(k)}(D,\boldsymbol w)$ for all second-rank tensor functions $\boldsymbol f:D\to \boldsymbol w$, $\boldsymbol w\subset \mathbb{C}^{3\times 3} $.
\end{defi}
 
\begin{defi}\label{Def:EuklNorm}
We define the inner product of two vectors $x,y\in\mathbb C^n$ by
\begin{equation*}\langle x,y\rangle := x\cdot \overline{y} := \sum_{j=1}^n x_j \overline{y_j}\end{equation*}
with the induced norm $|x|:=\sqrt{x\cdot\overline{x}}$. 
\end{defi}

\begin{defi}
The norm for $F\in\mathrm C(D)$, $D\subset \mathbb R^n$ compact, is given by
\begin{equation*} \Vert F\Vert_{\mathrm C(D)} := \sup_{x\in D} \vert F(x)\vert. \end{equation*}
Analogously, the norms on $\mathrm{C}(D,\mathbb{C}^m)$ are defined by using $|F(x)|$ in the sense of Definition \ref{Def:EuklNorm}. 
\end{defi}

\noindent The following definitions contain the notations for
the spherical geometry used in the construction of our Slepian functions.
\begin{defi}\label{defivarepsilon}
The unit sphere $\Omega$ of the three-dimensional Euclidean space $\mathbb R^3$ is represented by
\begin{equation*} \Omega=\left\lbrace\left. x\in\mathbb R^3 \ \right\vert \ \vert x\vert=1\right\rbrace. \end{equation*}
We will use the local orthonormal basis given by
\begin{equation*}\xi(t,\varphi)=\varepsilon^r=
\begin{pmatrix}
\sqrt{1-t^2} \ \cos\varphi \\
\sqrt{1-t^2} \ \sin\varphi\\
t
\end{pmatrix}, \qquad \varepsilon^\varphi=
\begin{pmatrix}
-\sin\varphi\\
\cos\varphi\\
0
\end{pmatrix},
\qquad \varepsilon^t =
\begin{pmatrix}
-t \ \cos\varphi\\
-t \ \sin\varphi\\
\sqrt{1-t^2}
\end{pmatrix}.\end{equation*}
Here, $t\in\lbrack -1,1\rbrack$ is the polar distance and
$\varphi\in\lbrack 0,2\pi)$ denotes the longitude. The polar
distance is related to the latitude $\theta \in \lbrack -\pi/2,\pi/2\rbrack$ through the relationship
$t=sin(\theta)$. Note that $\varepsilon^r$ is radially outward,
$\varepsilon^\varphi$ eastward and $\varepsilon^t$ northward. For
$t=-1$, we obtain the South pole and for $t=1$ the North pole.
Furthermore, we define the unit sphere without the poles
\begin{equation*} \Omega_0:=\Omega \setminus \left\lbrace \xi=\xi(t,\varphi)\ \vert \ t=\pm 1\right\rbrace, \end{equation*}
where $\xi=\xi(t,\varphi)$ is the polar coordinate representation of $\xi\in\Omega$.
\end{defi}

\begin{defi}
We define the tensors
\begin{equation*}\boldsymbol{\mathrm i}_\mathrm{tan}(\xi) := \varepsilon^\varphi \otimes \varepsilon^\varphi + \varepsilon^t \otimes \varepsilon^t\end{equation*}
and
\begin{equation*}\boldsymbol{\mathrm j}_\mathrm{tan}(\xi) := \varepsilon^t \otimes \varepsilon^\varphi - \varepsilon^\varphi \otimes \varepsilon^t\end{equation*}
for $\xi=\xi(t,\varphi)\in\Omega$. This definition follows the construction presented by \cite{freeden}. 
\end{defi}

\noindent In the following, we use the shorthand form for differentiation
\begin{equation*} \partial_x:=\frac{\partial}{\partial x} \end{equation*}
to define some well-known Cartesian and spherical differential operators.

\begin{defi}
The gradient is defined by
\begin{equation*} \nabla_x:= \left(\partial_{x_i}\right)_{i=1,2,3}=\begin{pmatrix} \partial_{x_1}\\ \partial_{x_2}\\ \partial_{x_3}\end{pmatrix}\end{equation*}
and the Laplace operator by
\begin{equation*} \Delta_x:=\partial_{x_1}^2+\partial_{x_2}^2+\partial_{x_3}^2 \end{equation*}
\noindent for $x=(x_1, x_2, x_3)^\mathrm T \in D\subset \mathbb R^3$.
\end{defi}

\begin{defi}\label{nablaLdefi}
The surface gradient
\begin{equation*}\nabla_\xi^\ast := \varepsilon^\varphi \ \frac{1}{\sqrt{1-t^2}} \ \partial_\varphi + \varepsilon^t \ \sqrt{1-t^2} \ \partial_t\end{equation*}
and the surface curl gradient
\begin{equation*}L_\xi^\ast := - \varepsilon^\varphi \ \sqrt{1-t^2} \ \partial_t + \varepsilon^t \ \frac{1}{\sqrt{1-t^2}} \ \partial_\varphi\end{equation*}
for $\xi=\xi(t,\varphi)\in\Omega$ are differential operators on the sphere such that $\nabla_{r\xi}=\left( \xi \partial_r+\frac{1}{r} \ \nabla^\ast_\xi\right)$ for $r\in \mathbb R^+$, $\xi \in\Omega$ and $L_\xi^\ast=\xi\wedge\nabla_\xi^\ast$, where $\wedge$ represents the vector product in $\mathbb R^3$.
\newline
\newline
Moreover,
\begin{equation*}\Delta^\ast_\xi :=\partial_t \left( \left(1-t^2\right) \partial_t\right)+ \ \frac{1}{1-t^2} \ \partial^2_\varphi\end{equation*}
is the Beltrami operator such that $\Delta^\ast=\nabla^\ast \cdot \nabla^\ast = L^\ast \cdot L^\ast$ and $\Delta_{r\xi} = \frac{\partial^2}{\partial r^2}+\frac{2}{r}\frac{\partial}{\partial r}+\frac{1}{r^2}\Delta^\ast_\xi$.
\end{defi}

\noindent This enables us to formulate Green's second surface identity \cite{volker}. Note that all integrals we use are Lebesgue integrals.

\begin{theorem}\label{greenoriginal}
Green's second surface identity is given by
\begin{equation*}\int_\Gamma \left(F(\xi)\Delta^\ast_\xi G(\xi)-G(\xi)\Delta^\ast_\xi F(\xi)\right) \ \mathrm d\omega(\xi)= \int_{\partial\Gamma}\left(F(\xi)\frac{\partial}{\partial\nu(\xi)} G(\xi)-G(\xi)\frac{\partial}{\partial\nu(\xi)} F(\xi)\right) \ \mathrm d\sigma(\xi),\end{equation*}
where $F,G\in \mathrm C^{(2)}\left(\overline{\Gamma}\right)$, $\Gamma\subset \Omega$ with a sufficiently smooth boundary and $\nu$ is the outward unit normal vector field to $\partial\Gamma$.
\newline
\newline
As a special case, Green's second surface identity over the entire unit sphere leads to
\begin{equation*}\int_\Omega \left(F(\xi)\Delta^\ast_\xi G(\xi)-G(\xi)\Delta^\ast_\xi F(\xi)\right) \ \mathrm d\omega(\xi)=0\end{equation*}
for $F,G\in \mathrm C^{(2)}(\Omega)$.
\end{theorem}

\noindent The proof to Theorem~\ref{greenoriginal} can be found in \cite[p.\ 448]{AmEsch08}. The construction of scalar, vector, and tensor Slepian functions
requires norms based on inner products. The following Hilbert spaces will satisfy this requirement. 
\begin{defi}
 
For a (Lebesgue) measurable set $D \subset \mathbb{R}^n$, we denote
with $ \mathrm L^2(D, \mathbb{C}^m ) $ the Hilbert space of (equivalence classes of almost everywhere identical) functions $F:D\to \mathbb{C}^m $ with
\begin{equation*}\Vert F\Vert_2:= \Vert F\Vert_{\mathrm L^{ 2 }(D, \mathbb{C}^m )} :=\left(\int_D \vert F(x)\vert^2 \ \mathrm dx\right)^\frac{1}{2}<\infty.\end{equation*}
\end{defi}

\noindent We focus our attention on the cases $m=1,3,3\times 3$, where $m=1$ yields scalar, $m=3$ vectorial, and $m=3\times 3$ tensorial function spaces. Consequently, we define $\mathrm L^2(D, \mathbb{C} )=:\mathrm L^2(D)$, $\mathrm L^2(D, \mathbb{C}^3 )=:\mathrm l^2(D)$, and $\mathrm L^2(D, \mathbb{C}^{3\times 3} )=:\boldsymbol{\mathrm l}^2(D)$. The inner products are denoted by
\begin{equation*}\left\langle F,G\right\rangle_{\mathrm L^2( D )} := \int_{ D } F(\xi) \overline{G(\xi)} \ \mathrm d\omega(\xi)\end{equation*}
for $m=1$, by
\begin{equation*}\left\langle f,g\right\rangle_{\mathrm l^2( D )} := \int_{ D } f(\xi) \cdot \overline{g(\xi)} \ \mathrm d\omega(\xi)\end{equation*}
for $m=3$, and by
\begin{equation*}\left\langle \boldsymbol f,\boldsymbol g\right\rangle_{\boldsymbol{\mathrm l}^2( D )} := \int_{ D } \boldsymbol f(\xi) : \overline{\boldsymbol g(\xi)} \ \mathrm d\omega(\xi)\end{equation*}
for $m=3\times 3$.

\begin{theorem}\label{LpC2}
The relationships between the Hilbert spaces and the spaces of continuous functions are 
\begin{eqnarray*}
 \overline{\mathrm C(\Omega)}^{\Vert \cdot
 \Vert_{\mathrm L^2(\Omega)}} &=& \mathrm L^2(\Omega),\\
 \overline{\mathrm c(\Omega)}^{\Vert \cdot
 \Vert_{ \mathrm{l}^2 (\Omega)}} &=& \mathrm l^2(\Omega),\\
 \overline{\boldsymbol{\mathrm c}(\Omega)}^{\Vert \cdot
 \Vert_{ \boldsymbol{\mathrm l}^2 (\Omega)}} &=& \boldsymbol{\mathrm{l}}^2(\Omega).
\end{eqnarray*}
\end{theorem}
\noindent The proof to Theorem~\ref{LpC2} can be found in \cite[Theorems 3.2.6 and 3.6.2]{VoigtWloka1975}.

\subsection{Scalar, Vector, and Tensor Spherical Harmonics}
 
Previously published constructions of scalar, vector, and tensor Slepian functions are based on classical spherical harmonics. Our unified approach for constructing Slepian functions is based on spin-weighted spherical harmonics. This approach will allow us to handle scalar, vectorial, and tensorial functions with the same setup. The reason is that spin-weighted spherical harmonics provide us with an equivalent representation of scalar, vector, and tensor spherical harmonics. Before explaining this in detail, we will first recapitulate a classical definition of the latter functions within this section.

\begin{defi} \label{spherharm}
We denote the (scalar) fully normalized spherical harmonics by
\begin{equation*}Y_{n,j}(\xi(t,\varphi)):= X_{n,j}(t) e^{i j\varphi} :=
\begin{cases}
(-1)^j \sqrt{\frac{2n+1}{4\pi}}\sqrt{\frac{(n-j)!}{(n+j)!}}\ P_{n,j}(t)e^{ij\varphi}, & j\geq 0\\
(-1)^j \ \overline{Y_{n,-j}(\xi(t,\varphi))}, & j<0
\end{cases},\end{equation*}
with the fully normalized associated Legendre functions given by
\begin{equation*}X_{n,j}(t) :=
\begin{cases}
(-1)^j \sqrt{\frac{2n+1}{4\pi}}\sqrt{\frac{(n-j)!}{(n+j)!}}\ P_{n,j}(t), & j\geq 0\\
(-1)^j X_{n,-j}(t), & j<0
\end{cases},\end{equation*}
the associated Legendre functions by
\begin{equation*}P_{n,j}(t):= \left(1-t^2\right)^{\frac{j}{2}} \left( \frac{\mathrm d}{\mathrm dt} \right)^j P_n(t),\end{equation*}
and the Legendre polynomials given by the Rodriguez formula
\begin{equation*}P_n(t) := \frac{1}{2^n n!} \ \left( \frac{\mathrm d}{\mathrm dt} \right)^n \left(t^2-1\right)^n, \end{equation*}
where $t\in \lbrack -1,1 \rbrack$, $\varphi\in[0,2\pi]$, $\xi \in \Omega$, and $n\in \mathbb N_0$, $j=-n,\dots,n$.
\end{defi}

\noindent The set of fully normalized spherical harmonics constructs a basis of $\left( \mathrm L^2(\Omega), \langle \cdot,\cdot \rangle_{\mathrm L^2(\Omega)} \right)$, see e.g.~\cite{volker} for a proof.
\newline
\newline
The following sets of functions form basis systems for $\left( \mathrm l^2(\Omega), \langle \cdot,\cdot \rangle_{\mathrm l^2(\Omega)} \right)$ and $\left( \boldsymbol{\mathrm l}^2(\Omega), \langle \cdot,\cdot \rangle_{\boldsymbol{\mathrm l}^2(\Omega)} \right)$, respectively, see e.g.~\cite{freeden2} for a proof.

\begin{defi}\label{vectorharm}
The vector spherical harmonics by Hill \cite{hill} (also called the Morse-Feshbach vector spherical harmonics, see \cite{morse}) are defined by
\begin{align*}
y_{n,j}^{(1)}(\xi)&:= \xi Y_{n,j}(\xi),\\
y_{n,j}^{(2)}(\xi)&:= \frac{1}{\sqrt{n(n+1)}} \ \nabla_\xi^\ast Y_{n,j}(\xi),\\
y_{n,j}^{(3)}(\xi)&:= \frac{1}{\sqrt{n(n+1)}} \ L_\xi^\ast Y_{n,j}(\xi)
\end{align*}
for $\xi \in \Omega$, $n\in\mathbb N_0$, $n\geq 0_i$, $j=-n,\dots,n$, and $i=1,2,3$ with
\begin{equation*} 0_i:=\begin{cases} 0, & i=1\\ 1, & i=2,3\end{cases}.\end{equation*}
 
\noindent We define the function spaces 
\begin{equation*}
 \mathrm{harm}_{n}(\Omega) := \mathrm{span}\, \left\lbrace y_{n,j}^{(i)} \Bigl\vert 1\leq i\leq 3 \text{ and } j=-n,\ldots,n \right\rbrace
\end{equation*}
and the function spaces
\begin{equation*}
 \mathrm{harm}_{p,\ldots,q}(\Omega) := \bigoplus_{n=p}^q \mathrm{harm}_{n}(\Omega)
\end{equation*}
for $p\leq q$.
\end{defi}

\noindent Note that the components normal to the unit sphere $\Omega$ are described by $y_{n,j}^{(1)}$, while $y_{n,j}^{(2)}$ and $y_{n,j}^{(3)}$ describe the tangential components. 

\begin{defi}\label{tensorharm}
The tensor spherical harmonics by Freeden, Gervens, and Schreiner \cite{freedenpaper} are defined by
\allowdisplaybreaks
\begin{align*}
\boldsymbol y_{n,j}^{(1,1)}(\xi)&:=\left(\xi \otimes \xi \right) Y_{n,j}(\xi),\\
\boldsymbol y_{n,j}^{(1,2)}(\xi)&:=\frac{1}{\sqrt{n(n+1)}} \ \left( \xi \otimes \nabla_\xi^\ast Y_{n,j}(\xi)\right),\\
\boldsymbol y_{n,j}^{(1,3)}(\xi)&:=\frac{1}{\sqrt{n(n+1)}} \ \left( \xi \otimes L_\xi^\ast Y_{n,j}(\xi)\right),\\
\boldsymbol y_{n,j}^{(2,1)}(\xi)&:=\frac{1}{\sqrt{n(n+1)}} \ \left( \nabla_\xi^\ast Y_{n,j}(\xi) \otimes \xi \right),\\
\boldsymbol y_{n,j}^{(2,2)}(\xi)&:=\frac{1}{\sqrt{2}} \ \boldsymbol{\mathrm i}_\mathrm{tan}(\xi) Y_{n,j}(\xi) \\
&=\frac{1}{\sqrt{2}} \ Y_{n,j}(\xi) \left( \varepsilon^\varphi \otimes \varepsilon^\varphi + \varepsilon^t \otimes \varepsilon^t \right),\\
\boldsymbol y_{n,j}^{(2,3)}(\xi)&:=\frac{1}{\sqrt{2n(n+1)(n(n+1)-2)}} \ \left\lbrack \left( \nabla_\xi^\ast \otimes \nabla_\xi^\ast -L_\xi^\ast \otimes L_\xi^\ast \right )Y_{n,j}(\xi) +2 \nabla_\xi^\ast Y_{n,j}(\xi) \otimes \xi \right\rbrack,\\
\boldsymbol y_{n,j}^{(3,1)}(\xi)&:=\frac{1}{\sqrt{n(n+1)}} \ \left( L_\xi^\ast Y_{n,j}(\xi) \otimes \xi \right),\\
\boldsymbol y_{n,j}^{(3,2)}(\xi)&:=\frac{1}{\sqrt{2n(n+1)(n(n+1)-2)}} \ \left\lbrack \left( \nabla_\xi^\ast \otimes L_\xi^\ast +L_\xi^\ast \otimes \nabla_\xi^\ast \right )Y_{n,j}(\xi) +2 L_\xi^\ast Y_{n,j}(\xi) \otimes \xi \right\rbrack,\\
\boldsymbol y_{n,j}^{(3,3)}(\xi)&:=\frac{1}{\sqrt{2}} \ \boldsymbol{\mathrm j}_\mathrm{tan}(\xi) Y_{n,j}(\xi)\\
&=\frac{1}{\sqrt{2}} \ Y_{n,j}(\xi) \left( \varepsilon^t \otimes \varepsilon^\varphi - \varepsilon^\varphi \otimes \varepsilon^t \right)
\end{align*}
for $\xi \in \Omega$, $n\in\mathbb N_0$, $n\geq 0_{ik}$, $j=-n,\dots,n$, and $i,k=1,2,3$ with
\begin{equation*} 0_{ik}:=\begin{cases} 0, & (i,k)=(1,1),(2,2),(3,3)\\ 1, & (i,k)=(1,2),(1,3),(2,1),(3,1)\\ 2, & (i,k)=(2,3),(3,2)\end{cases}.\end{equation*}

\noindent As in \cite{freeden}, we define the function spaces 
\begin{equation*} 
 \boldsymbol{\mathrm{harm}}_{n}(\Omega) := \mathrm{span}\, \left\lbrace \boldsymbol y_{n,j}^{(i,k)} \Bigl\vert 1\leq i,k\leq 3 \text{ and } j=-n,\ldots,n \right\rbrace,
\end{equation*}
and the function spaces
\begin{equation*}
 \boldsymbol{\mathrm{harm}}_{p,\ldots,q}(\Omega) := \bigoplus_{n=p}^q \boldsymbol{\mathrm{harm}}_{n}(\Omega),
\end{equation*}
for $p\leq q$.
\end{defi}

\noindent With these definitions, $\boldsymbol y_{n,j}^{(1,1)}$ is normal, $\boldsymbol y_{n,j}^{(1,2)}$, $\boldsymbol y_{n,j}^{(1,3)}$ are left normal/right tangential, $\boldsymbol y_{n,j}^{(2,1)}$, $\boldsymbol y_{n,j}^{(3,1)}$ are left tangential/ right normal, and $\boldsymbol y_{n,j}^{(2,2)}$, $\boldsymbol y_{n,j}^{(2,3)}$, $\boldsymbol y_{n,j}^{(3,2)}$, $\boldsymbol y_{n,j}^{(3,3)}$ are tangential.

\subsection{Spin-Weighted Spherical Harmonics}

It is known that scalar, vector, and tensor spherical harmonics can be used to construct spherical Slepian functions (see e.g.~\cite{eshagh} for the tensorial case). In addition, commuting operators are known for the scalar and the vector Slepian case for spherical-cap regions (\cite{grunbaum, jahn}). The advantage of the spin-weighted spherical harmonics by Newman and Penrose \cite{newmanpenrose} (which we introduce in this section) over the classical spherical harmonics is that the former allow us to (i) derive a unified approach to construct Slepian functions for arbitary tensor ranks and (ii) construct commuting operators for spherical-cap regions for arbitary tensor ranks (see also \cite{handbook, diss}). 

\begin{defi}\label{defieth}
 Following the construction of \cite{newmanpenrose}, we define the spin-weighted differential operators $\eth_N:\mathrm C^{(1)}(\Omega_0)\to \mathrm C(\Omega_0)$ and $\overline \eth_N:\mathrm C^{(1)}(\Omega_0)\to \mathrm C(\Omega_0)$ of spin weight $N\in\mathbb Q$ by
\begin{align*}
\eth_N \ F(\xi)&:= \left( \sqrt{1-t^2} \ \partial_t + \frac{Nt-i \partial_\varphi}{\sqrt{1-t^2}}\right)F(\xi),\\
\overline{\eth}_N \ F(\xi)&:= \left( \sqrt{1-t^2} \ \partial_t - \frac{Nt-i \partial_\varphi}{\sqrt{1-t^2}}\right)F(\xi),
\end{align*}
where $\xi=\xi(t,\varphi)\in\Omega_0$ and $F\in \mathrm C^{(1)}(\Omega_0)$.
\end{defi}

\begin{defi}
 Symbol $\eth^M_N$ for $M\in \mathbb Q^+_0$ and $N\in\mathbb Q$ denotes the successive application of spin-weighted operators $\eth_k$ for spin weights $k=N, N+1, \ldots, N+M-1$ on a function $F\in\mathrm C^{(M)}(\Omega_0)$ such that
\begin{equation*}\eth^M_N \ F := \eth_{N+M-1} \eth_{N+M-2} \dots \eth_{N+1} \eth_N \ F.\end{equation*}
 Mutatis mutandis, we define
\begin{equation*}\overline{\eth}^M_N \ F := \overline{\eth}_{N-M+1} \overline{\eth}_{N-M+2} \dots \overline{\eth}_{N-1} \overline{\eth}_N \ F.\end{equation*}
The case $M=0$ denotes the identity operator 
\begin{equation*} \eth^0_N=\mathrm{Id}=\overline \eth^0_N. \end{equation*}
\end{defi}

\begin{defi}\label{newdef}
The spin-weighted spherical harmonics by Newman and Penrose \cite{newmanpenrose} (see also \cite{dray, goldberg, lewis, castillo, wiaux2, wiaux}) are defined for $n\in\mathbb N_0$, $N\in\mathbb Q$, $n\geq \vert N\vert$, and $j=-n,\dots,n$ by
\begin{equation*}{}_NY_{n,j} :=
\begin{cases}
\sqrt{\frac{(n-N)!}{(n+N)!}} \ \eth^N_0 \ Y_{n,j}, & 0\leq N\leq n\\
(-1)^N \sqrt{\frac{(n+N)!}{(n-N)!}} \ \overline{\eth}^{-N}_0 \ Y_{n,j}, & -n\leq N\leq 0\\
0,& n< \vert N\vert
\end{cases}.\end{equation*}
\end{defi}

\noindent Note that we define the spin-weighted spherical harmonics on $\Omega_0$, because the operators $\eth$ and $\overline\eth$ have singularities at the poles.
\newline
\newline
Lemma~\ref{defspin_allN} and Theorem~\ref{wigner} provide alternative formulations for the spin-weighted spherical harmonics. 

\begin{lemma}\label{defspin_allN}
The spin-weighted spherical harmonics fulfill for all $N\in\mathbb Z$, all $n\in\mathbb N_0$, and all $j=-n,\dots,n$ the following properties
\begin{equation} \eth_N \ {}_N Y_{n,j} = \sqrt{n(n+1)-N(N+1)} \ {}_{N+1} Y_{n,j}, \label{eqeth1}\end{equation}
and
\begin{equation} \overline{\eth}_N \ {}_N Y_{n,j} =-\sqrt{n(n+1)-N(N-1)} \ {}_{N-1} Y_{n,j}.\label{eqeth2} \end{equation}
\end{lemma}
\noindent See \cite{dray, goldberg, lewis, newmanpenrose, castillo, wiaux2, wiaux} for \ a \ proof. 

\begin{theorem}\label{wigner}
The spin-weighted spherical harmonics also satisfy \cite{lewis, wiaux2, wiaux}
\begin{align*}
{}_{N} Y_{n,j}(\xi) &= (-1)^N\sqrt{\frac{2n+1}{4\pi}} \ e^{ij\varphi} d_{j,-N}^n(\vartheta)\\
&= (-1)^N\sqrt{\frac{2n+1}{4\pi}} \ \overline{D_{j,-N}^n(\varphi, \vartheta,0)},
\end{align*}
where $\xi=\xi(t,\varphi)\in\Omega_0$, $t=\cos \vartheta$, $n\in\mathbb N_0$, $N\in\mathbb Z$, $n\geq\vert N\vert$, $j=-n,\dots,n$, and $D_{j,N}^n$ is the Wigner $D$-function. 
\end{theorem}
\noindent See \cite{edmonds, wigner} for a proof. 
\newline
\newline
\noindent Furthermore, the spin-weighted spherical harmonics can also be formulated as functions of spin weight. First, we define a function of spin weight $N$.

\begin{defi}
The coefficients $d_{i_1i_2 \dots i_{2n}} \in \mathbb R$ for $n\in\mathbb N_0$ are called totally symmetric, if they are equal for every permutation of the index $(i_1,i_2, \dots,i_{2n})\in\mathbb N_0^{2n}$.
\end{defi}

\begin{defi}\label{entwspin}
A function ${}_NF_n \in \mathrm L^2(\Omega)$ is called a function of spin weight $N\in \mathbb Z$ and degree $n\in\mathbb N_0$, if it can be written as \cite{castillo}
\begin{equation*}{}_NF_n = \sum_{i_1,\dots,i_{2n}=1}^2 d_{i_1i_2\dots i_{2n}} \underbrace{o^{i_1}o^{i_2}\dots o^{i_{n+N}}}_{n+N} \underbrace{\hat o^{i_{n+N+1}}\dots \hat o^{i_{2n}}}_{n-N},\end{equation*}
where $\vert N\vert \leq n$, the coefficients $d_{i_1i_2\dots i_{2n}}\in\mathbb R$ are totally symmetric, and for $\xi=\xi(t,\varphi)\in\Omega$, we define
\begin{align*}
o^1(\xi)&:=o^1_\xi:=e^{-i\frac{\varphi}{2}} \sqrt{\frac{1+t}{2}}, & o^2(\xi)&:=o^2_\xi:=e^{i\frac{\varphi}{2}} \sqrt{\frac{1-t}{2}},\\
\hat o^1(\xi)&:=\hat o^1_\xi:=-e^{-i\frac{\varphi}{2}} \sqrt{\frac{1-t}{2}}, &\hat o^2(\xi)&:=\hat o^2_\xi:=e^{i\frac{\varphi}{2}} \sqrt{\frac{1+t}{2}}.
\end{align*}
\end{defi}

\begin{lemma}\label{spinino}
The spin-weighted spherical harmonics can be represented as functions of spin weight $N\in\mathbb Z$ for $\xi=\xi(t,\varphi)\in\Omega$ by
\begin{align*}
{}_NY_{n,j}(\xi) &= (-1)^j \sqrt{\frac{2n+1}{4\pi}} \sqrt{(n-j)!(n+j)!(n-N)!(n+N)!}\\
&\quad \times \sum_{k=\max\lbrace 0,j-N\rbrace}^{\min\lbrace n+j,n-N\rbrace} \frac{\left(o^1_\xi\right)^{k+N-j}\left(o^2_\xi\right)^{n-k+j}\left(\hat o^1_\xi\right)^{n-k-N}\left(\hat o^2_\xi\right)^k}{k!(n+j-k)!(n-N-k)!(N-j+k)!},
\end{align*}
where $n\in\mathbb N_0$, $n\geq\vert N\vert$, and $j=-n,\dots,n$.
\end{lemma}

\noindent Similar to the classical spherical harmonics, the spin-weighted spherical harmonics also satisfy recursion relations, a Christoffel-Darboux formula, and an addition theorem. These have been proven in \cite{handbook, diss} for the first time. In the following, we report the results and refer to \cite{handbook} and \cite{diss} for the proofs. 

\begin{theorem}\label{recspin}
The spin-weighted spherical harmonics satisfy the following recursion relations for $\xi=\xi(t,\varphi)\in\Omega_0$:
\begin{align}
\left(t^2-1\right) \partial_t \ {}_N Y_{n,j}(\xi) &= \left(nt + \frac{Nj}{n} \right) {}_N Y_{n,j}(\xi) -(2n+1) \alpha^N_{n,j} \ {}_N Y_{n-1,j}(\xi), \label{rec1}\\
&= -\left( (n+1)t + \frac{Nj}{n+1} \right) {}_NY_{n,j}(\xi) +(2n+1) \alpha^N_{n+1,j} \ {}_N Y_{n+1,j}(\xi), \label{rec2}\\
\left( t + \frac{Nj}{n(n+1)} \right) {}_N Y_{n,j}(\xi) &= \alpha^N_{n,j} \ {}_N Y_{n-1,j}(\xi) + \alpha^N_{n+1,j} \ {}_N Y_{n+1,j}(\xi), \label{rec3} \end{align}
where
\begin{equation*}\alpha^N_{n,j}:= \frac{\sqrt{(n-N)(n+N)}}{n} \ c_{n,j}=\frac{\sqrt{(n-N)(n+N)}}{n} \sqrt{\frac{(n-j)(n+j)}{(2n-1)(2n+1)}},\end{equation*}
$N\in\mathbb Z$, $n\in\mathbb N_0$, $n\geq\vert N+1\vert$, and $j=-n,\dots,n$. Furthermore, we denote ${}_NY_{n,j}:=0$ for $n<\vert j\vert$.
\end{theorem}

\begin{theorem}[Christoffel-Darboux Formula]\label{CDF}
For all $N \in \mathbb Z$, we obtain the Christoffel-Darboux formula for the spin-weighted spherical harmonics
\begin{equation*}\left(t_1-t_2\right) \ \sum_{n=n_j}^{L-1} \overline{{}_NY_{n,j}(\xi)} \ {}_NY_{n,j}(\eta) = \alpha^N_{L,j} \left( \overline{{}_NY_{L,j}(\xi)}\ {}_NY_{L-1,j}(\eta) - \overline{{}_NY_{L-1,j}(\xi)} \ {}_NY_{L,j}(\eta) \right),\end{equation*}
where
\begin{equation*}n_j:=\max\lbrace \vert N \vert, \vert j \vert \rbrace,\end{equation*}
$\xi=\xi(t_1,\varphi_1)$, $\eta=\eta(t_2,\varphi_2)$ are the polar coordinate representations of $\xi,\eta\in\Omega_0$, $L>n_j$ is the bandlimit, and $j=-L,\dots,L$.
\end{theorem}

\begin{theorem}\label{orth}
The spin-weighted spherical harmonics are orthonormal with respect to the $\mathrm L^2(\Omega)$-inner product,
\begin{equation*}\int_\Omega {}_NY_{n,j}(\xi) \ \overline{{}_NY_{n',j'}(\xi)} \ \mathrm d\omega(\xi) = \delta_{n,n'} \delta_{j,j'}.\end{equation*}
\end{theorem}
\noindent See \cite{dahlentromp, dray, goldberg, hu, lewis, newmanpenrose, wiaux} for the proof of Theorem~\ref{orth}. 

\begin{theorem}[Addition Theorem for Spin-Weighted Spherical Harmonics] \label{addtheotheo}
The spin-weighted spherical harmonics satisfy the following addition theorem for $N_1,N_2\in\mathbb Z$ and for $n\in\mathbb N_0$, $n\geq\max\lbrace \vert N_1\vert, \vert N_2\vert\rbrace$,
\begin{equation*}\sum_{j=-n}^n {}_{N_1}Y_{n,j}(\xi_1) \ \overline{{}_{N_2}Y_{n,j}(\xi_2)} = (-1)^{N_1} \sqrt{\frac{2n+1}{4\pi}} \ {}_{N_2}Y_{n,-N1}(\xi) \ e^{-iN_2\gamma},\end{equation*}
where $\xi_1=\xi_1(t_1,\varphi_1)$, $\xi_2=\xi_2(t_2,\varphi_2)$, $t_i=\cos\vartheta_i$, $i=1,2$, $\xi=\xi(t,\alpha)\in\Omega$, $t=\cos \beta$, and $\alpha$, $\beta$, and $\gamma$ are the Euler angles given by
\begin{itemize}
\item for $\sin (\varphi_1-\varphi_2) \neq 0$
\begin{align*}
\cot \alpha &= \cos \vartheta_1 \cot (\varphi_1-\varphi_2) - \cot \vartheta_2 \frac{\sin \vartheta_1}{\sin(\varphi_1-\varphi_2)},\\
\cos \beta &= \cos \vartheta_1 \cos \vartheta_2 + \sin \vartheta_1 \sin \vartheta_2 \cos( \varphi_1-\varphi_2),\\
\cot \gamma &= \cos \vartheta_2 \cot (\varphi_1-\varphi_2) - \cot \vartheta_1 \frac{\sin \vartheta_2}{\sin (\varphi_1-\varphi_2)}.
\end{align*}
\item for $\sin (\varphi_1-\varphi_2) = 0$, so $\varphi_1-\varphi_2 = k\pi$, $k\in\mathbb Z$, then
\begin{equation*}\begin{rcases}
\alpha=\pi, \beta=\vartheta_1-\vartheta_2, \gamma=\pi &, \text{ if } k \text{ even}, -\vartheta_1+\vartheta_2\in \lbrack -\pi,0)\\
\alpha=0, \beta=-\vartheta_1+\vartheta_2, \gamma=0 &, \text{ if }k \text{ even}, -\vartheta_1+\vartheta_2\in \lbrack 0,\pi)\\
\alpha=\pi, \beta=\vartheta_1+\vartheta_2, \gamma=0 &, \text{ if } k \text{ odd}, \vartheta_1+\vartheta_2\in\lbrack 0,\pi)\\
\alpha=0, \beta=2\pi-(\vartheta_1+\vartheta_2), \gamma=\pi &, \text{ if } k \text{ odd}, \vartheta_1+\vartheta_2\in\lbrack \pi,2\pi)
\end{rcases}.\end{equation*}
\end{itemize}
\end{theorem}

\noindent The proof can be found in \cite{handbook,diss}.

\begin{fol}\label{addtheo}
With $\xi_1=\xi_2=\eta\in\Omega$ and with $N_1=N_2=N$ the addition theorem reduces to
\begin{equation*} \sum_{j=-n}^n {}_NY_{n,j}(\eta) \ \overline{{}_NY_{n,j}(\eta)} = \frac{2n+1}{4\pi}. \end{equation*}
\end{fol}

\begin{theorem}\label{dglfol}
 The spin-weighted spherical harmonics are eigenfunctions of the spin-weighted Beltrami operator 
 \begin{equation*}\Delta^{\ast,N}_\xi := \Delta^\ast_\xi - \frac{N^2-2itN\partial_\varphi}{1-t^2},\end{equation*}
where the classical Beltrami operator \ satisfies \ $\Delta^\ast_\xi = \partial_t
\left((1-t^2)\partial_t\right)+\frac{1}{1-t^2}
\ \partial_\varphi^2$. Hence, for all $\xi=\xi(t,\varphi)\in\Omega_0$,
all $N \in \mathbb Z$, all $n\in\mathbb N_0$, $n\geq\vert N\vert$, and
all $j=-n,\dots,n$
\begin{equation*}\Delta^{\ast,N}_\xi \ {}_N Y_{n,j}(\xi) = -n(n+1) \ {}_N Y_{n,j}(\xi).\end{equation*}
\end{theorem}

\noindent The spin-weighted spherical harmonics are part of a function space with a series of properties that we will need in the construction of the commuting operator for the spin-weighted Slepian functions.

\begin{defi}\label{defspacecomm}
We denote by $\mathrm X^k(\Gamma)$, $k\in\mathbb N_0$, the set of all functions $F\in\mathrm C^{(k)}(\Gamma)\cap\mathrm L^2\left(\overline{\Gamma}\right)$ which satisfy the following conditions, where $\xi=\xi(t,\varphi)\in\Gamma\subset\Omega$:
\begin{itemize}
\item $F$ has the form $H(t)e^{ij\varphi}$ for $j\in\mathbb Z$,
\item $F$ is bounded on $\overline\Gamma$,
\item $\partial_t^{ j } F(\xi) = \mathcal O\left(\left(1-t^2\right)^{\frac{1}{2}- j }\right)$ as $t\to\pm 1$ \ for all $j=0,\dots,k$,
\item $\Delta^{\ast,N}_\xi F(\xi) = \mathcal O(1)$ for $N\in\mathbb Z$ as $t\to\pm 1$.
\end{itemize}
\end{defi}

\begin{fol}\label{folYinspace}
The previous definition is chosen such that
\begin{equation*} {}_NY_{n,j}\in \mathrm X^k(\Omega_0) \end{equation*}
with $\overline{\Omega_0}=\Omega$, for all $n\in\mathbb N_0$, $N\in\mathbb Z$, $n\geq\vert N\vert$, $j=-n,\dots,n$, and all $k\in\mathbb N_0$.
\end{fol}
\noindent See \cite{handbook, diss} for the proof of Corollary~\ref{folYinspace}.

\begin{theorem}[Green's Second Surface Identity for the Spin-Weighted Beltrami Operator]\label{green}
 Let $\Gamma\subset\Omega$ with a sufficiently smooth boundary $\partial\Gamma$. For $F,G\in\mathrm X^2\left(\overline\Gamma\right)$ 
\begin{align*}
&\int_\Gamma \left(F(\xi)\overline{\Delta^{\ast,N}_\xi G(\xi)}-\overline{G(\xi)}\Delta^{\ast,N}_\xi F(\xi)\right) \ \mathrm d\omega(\xi)\\
&=\int_{\partial\Gamma}\left(F(\xi)\frac{\partial}{\partial\nu(\xi)} \overline{G(\xi)}-\overline{G(\xi)}\frac{\partial}{\partial\nu(\xi)} F(\xi)\right) \ \mathrm d\sigma(\xi)-\int_\Gamma \frac{2iNt}{1-t^2} \ \partial_\varphi \left(F(\xi)\overline{G(\xi)}\right) \ \mathrm d\omega(\xi),
\end{align*}
if the integrals exist.
\newline
\newline
 As a special case, Green's second surface identity for the spin-weighted Beltrami operator $\Delta^{\ast,N}$ over the entire unit sphere yields
\begin{equation*}\int_\Omega \left(F(\xi)\overline{\Delta^{\ast,N}_\xi G(\xi)}-\overline{G(\xi)}\Delta^{\ast,N}_\xi F(\xi)\right) \ \mathrm d\omega(\xi)=0\end{equation*}
and for the region $R$ a polar cap, 
\begin{align*}
&\int_R \left(F(\xi)\overline{\Delta^{\ast,N}_\xi G(\xi)}-\overline{G(\xi)}\Delta^{\ast,N}_\xi F(\xi)\right) \ \mathrm d\omega(\xi)\\
&=\int_0^{2\pi} \left\lbrack \left(1-t^2\right)\left(\overline{G(\xi)}\partial_t F(\xi)-F(\xi)\partial_t\overline{G(\xi)}\right)\right\rbrack_{t=b} \ \mathrm d\varphi
\end{align*}
\noindent for $F,G\in \mathrm X^2(\Omega_0)$, where $\xi=\xi(t,\varphi)\in\Omega$.
\end{theorem}

\noindent See \cite{handbook, diss} for a proof of Theorem~\ref{green}. It can also be shown that the spin-weighted spherical harmonic functions form a complete orthonormal system for $\mathrm{L}^2(\Omega)$, see e.g.\ \cite{handbook,diss} for a proof.

\begin{theorem}\label{only}
The set of functions $\left\lbrace {}_NY_{n,j}\right\rbrace_{n\geq \vert N\vert, j=-n,\dots,n}$ forms a complete orthonormal system for $\left(\mathrm L^2(\Omega),\langle \cdot,\cdot\rangle_{\mathrm L^2(\Omega)}\right)$. Consequently, for $F\in\mathrm L^2(\Omega)$, we obtain
\begin{equation*} \lim_{L\to \infty} \left\Vert F-\sum_{n=\vert N\vert}^L\sum_{j=-n}^n \langle F,{}_NY_{n,j}\rangle_{\mathrm L^2(\Omega)} \ {}_NY_{n,j}\right\Vert_{\mathrm L^2(\Omega)} =0. \end{equation*}
 Hence, for every function $F\in \mathrm L^2(\Omega)$ and for every $N\in\mathbb Z$, there exist unique coefficients ${}_NF_{n,j}$ such that
\begin{equation*}F=\sum_{n=\vert N\vert}^\infty \sum_{j=-n}^n {}_NF_{n,j} \ {}_NY_{n,j}\end{equation*}
 in the sense of $\mathrm{L}^2(\Omega)$. Furthermore, for every function $F,G\in\mathrm L^2(\Omega)$ and every $N\in\mathbb Z$, the spin-weighted spherical harmonics satisfy the Parseval identity
\begin{equation*}\langle F,G\rangle_{\mathrm L^2(\Omega)} = \sum_{n=\vert N\vert}^\infty \sum_{j=-n}^n \left\langle F,{}_NY_{n,j}\right\rangle_{\mathrm L^2(\Omega)} \overline{\left\langle G,{}_NY_{n,j}\right\rangle_{\mathrm L^2(\Omega)}}\end{equation*}
and consequently,
\begin{equation*}\Vert F\Vert^2_{\mathrm L^2(\Omega)} = \sum_{n=\vert N\vert}^\infty \sum_{j=-n}^n \left\vert \left\langle F,{}_NY_{n,j}\right\rangle_{\mathrm L^2(\Omega)}\right\vert^2.\end{equation*}
\end{theorem}

\noindent Analogously to the classical harmonic function spaces (see e.g.~\cite{freeden}), we can now define the function spaces for the spin-weighted spherical harmonics. 

\begin{defi}
A function $Y_n\in \mathrm C^{(2)}(\Omega_0)$ of degree $n\in\mathbb N_0$ is called $(\ast,N)$-harmonic for $N\in\mathbb Z$, if for all $\xi\in\Omega_0$
\begin{equation*}\Delta_\xi^{\ast,N} \ Y_n(\xi)=-n(n+1) Y_n(\xi).\end{equation*}
\end{defi}

\begin{defi}
With $\mathrm{Harm}_n^N(\Omega_0)$, we denote the set of the $(\ast,N)$-harmonic functions of degree $n\in\mathbb N_0$ and spin weight $N\in\mathbb Z$. This means that
\begin{align*}
\mathrm{Harm}_n^N(\Omega_0):&=\left\lbrace {}_NP_n \ \Bigl\vert \ {}_NP_n\in\mathrm C^{(2)}(\Omega_0) \text{ is a } (\ast,N)\text{-harmonic function of spin} \text{ weight } N \right.\\
&\qquad\qquad\qquad\qquad\qquad\qquad\qquad\qquad\qquad\qquad\qquad\qquad\qquad\text{ and degree } n \Bigr\rbrace.
\end{align*}
\end{defi}

\begin{defi}
 Analogously to the (spin-free) notations (cf.~Definition \ref{vectorharm} and \cite{volker}), we define the spaces
\begin{equation*}\mathrm{Harm}^N_{0\dots n}(\Omega_0):= \bigoplus_{i=0}^n \mathrm{Harm}^N_i(\Omega_0)\end{equation*}
and
\begin{equation*}\mathrm{Harm}^N_{0\dots \infty}(\Omega_0):= \bigcup_{i=0}^\infty \mathrm{Harm}^N_{ i }(\Omega_0).\end{equation*}
\end{defi}

\noindent Note that multiplying a $(\ast,N)$-harmonic function by $r^n$, $r\in\mathbb R$, $n\in\mathbb N_0$ does \textit{not} necessarily render it a harmonic function. Similarly, a function of spin weight $N\neq 0$ is generally \textit{not} homogeneous. Only for the special case of spin weight zero do we get $\mathrm{Harm}_n^N(\Omega_0)=\mathrm{Harm}_n(\Omega_0)$, the set of the to the unit sphere restricted harmonic and homogeneous polynomials of degree $n\in\mathbb N_0$.

\begin{fol}\label{cor_functioninharm}
The spin-weighted spherical harmonics ${}_NY_n$ of spin weight $N\in\mathbb Z$ and degree $n\in\mathbb N_0$, $n\geq\vert N\vert$, span the set $\mathrm{Harm}_n^N(\Omega_0)$. The functions $\left\lbrace {}_NY_{n,j}\right\rbrace_{j=-n,\dots,n}$ form an orthonormal system in \ the space $\left( \mathrm{Harm}_n^N(\Omega_0),\langle \cdot,\cdot \rangle_{\mathrm L^2(\Omega)}\right)$.
\end{fol}

\begin{remark}\label{Rem:HarmpqN}
For $ |N| \leq p \leq q \leq \infty$, we define 
\begin{equation*} \mathrm{Harm}^N_{p\dots q}(\Omega_0) := \mathrm{span} \left\lbrace {}_NY_{n,j}\right\rbrace_{n=p,\dots,q,j=-n,\dots,n}.\end{equation*}
\end{remark}

\begin{theorem}\label{eigenfuncs}
 For all $N\in\mathbb Z$, all $n\in\mathbb N_0$, $n\geq\vert N\vert$, and all $j=-n,\dots,n$, the spin-weighted spherical harmonics ${}_NY_{n,j}$ and their linear combinations are the only eigenfunctions of the differential operator $\Delta^{\ast,N}$ in $\mathrm X^2(\Omega_0)$. Their eigenvalues are $-n(n+1)$.
\end{theorem}
\noindent See \cite{handbook, diss} for the proof of Theorem~\ref{eigenfuncs}. The following theorem allows us to relate the spin-weighted spherical harmonics to the classical scalar, vector, and tensor harmonics. 

\begin{theorem}\label{relation_tens23}
 The spin-weighted spherical harmonics multiplied by unit
 vectors $\tau_\pm$ or unit tensors $\tau_\pm\otimes \xi$, $\xi \otimes \tau_\pm$, or $\tau_\pm\otimes \tau_\pm$, for
\begin{equation*}\tau_\pm := -\frac{1}{\sqrt{2}} \ \left(\varepsilon^t \pm i \varepsilon^\varphi\right) \quad\text{ and }\quad \varepsilon^r = \xi \in \Omega \end{equation*}
 are linear combinations of classical
 scalar, vector, and tensor spherical harmonics in the following way. 
\begin{itemize}
\item The scalar spherical harmonics are equal to the spin-weighted spherical harmonics of spin weight 0, 
 \begin{equation*}{}_0 Y_{n,j} = Y_{n,j}.\end{equation*}
\item Spin-weighted spherical harmonics of spin weight~$\pm 1$, multiplied by the unit vectors $\tau_\pm$ are linear combinations of the classical tangential vector spherical harmonics, 
\begin{equation*}\pm \frac{1}{\sqrt{2}} \ \left( -y_{n,j}^{(2)}(\xi) \pm i y_{n,j}^{(3)} (\xi)\right) = {}_{\pm 1} Y_{n,j}(\xi) \tau_\pm.\end{equation*}
\item Spin-weighted spherical harmonics of spin weight~$\pm 1$ multiplied by the unit tensor $\xi \otimes \tau_\pm$ are linear combinations of the left normal/right tangential tensor spherical harmonics, 
 \begin{equation*}\pm \frac{1}{\sqrt{2}} \ \left( -\boldsymbol y_{n,j}^{(1,2)}(\xi) \pm i \boldsymbol y_{n,j}^{(1,3)}(\xi) \right) = {}_{\pm 1} Y_{n,j}(\xi) \left( \xi \otimes \tau_\pm \right).\end{equation*}
 Spin-weighted spherical harmonics of spin weight~$\pm 1$ multiplied by the unit tensor $\tau_\pm \otimes \xi$ are linear combinations of the left tangential/right normal tensor spherical harmonics, 
 \begin{equation*}\pm \frac{1}{\sqrt{2}} \ \left( -\boldsymbol y_{n,j}^{(2,1)}(\xi) \pm i \boldsymbol y_{n,j}^{(3,1)}(\xi) \right) = {}_{\pm 1} Y_{n,j}(\xi) \left( \tau_\pm \otimes \xi \right).\end{equation*}
 Spin-weighted spherical harmonics of spin weight~$\pm 2$, multiplied by the unit tensors $\tau_\pm \otimes \tau_\pm$ are linear combinations of the tangential tensor spherical harmonics, 
\begin{equation*}-\frac{1}{\sqrt{2}} \ \left( -\boldsymbol y_{n,j}^{(2,3)}(\xi) \pm i \boldsymbol y_{n,j}^{(3,2)}(\xi) \right) = {}_{\pm 2} Y_{n,j}(\xi) \left( \tau_\pm \otimes \tau_\pm \right).\end{equation*}
\end{itemize}
\end{theorem}
\noindent The relationship for the scalar spherical harmonics follows directly from the definition of the spin-weighted spherical harmonics. For the remainder of the proof, see \cite{thorne}. Note that $\tau_\pm \cdot \overline{\tau_\pm} =1$ and $\tau_\pm \cdot \overline{\tau_\mp} =0$ for all $\xi\in\Omega$. 

\section{Spin-Weighted Slepian Functions}\label{Sec:SPWSlepFct}

Similar to the case for the classical spherical harmonics
presented in \cite{simonspaper}, we solve the concentration problem on
the unit sphere for bandlimited functions of spin weight
$N$. Section~\ref{scaslepderiv} solves the bandlimited optimization
problem for general regions by deriving an equivalent
finite-dimensional eigenvalue problem whose matrix is Hermitian and
positive definite. This eigenvalue problem is equivalent to a homogeneous integral equation of the second kind.\\
For the special case of polar cap regions and spin
weights $N\in \mathbb Z$, we find commuting operators
(Section~\ref{chapcomcap}) with corresponding kernel matrices that
are tridiagonal and which have simple eigenvalues
(Section~\ref{compu}). Because the eigenvectors (but not the
eigenvalues) of the constructed commuting matrices are equal to the
eigenvectors of the original kernel matrix, this reduces the numerical
cost of constructing the Slepian functions for polar caps and spin
weights $N\in \mathbb Z$, and increases the numerical
stability. Slepian functions for general spherical caps can be
constructed from those for a polar cap of equal opening angle through
rotation.

\begin{defi}
If it exists, then the bandlimit $L\in \mathbb{N}_0$ of a function ${}_N F \in L^2(\Omega)$ is the smallest non-negative integer such that ${}_N F \in \mathrm{Harm}^N_{\lvert N \rvert \dots L}(\Omega)$. If no such number exists, then $L=\infty$.
\end{defi}

\subsection{Construction for General Regions}\label{scaslepderiv} 

As a consequence of Theorem~\ref{only} and Remark \ref{Rem:HarmpqN}, every function ${}_NF\in \mathrm L^2(\Omega)$ of spin weight $N\in\mathbb Z$ that is bandlimited by $L$ (or ${}_NF\in \mathrm{Harm}^N_{\vert N\vert \dots L}(\Omega)$, for short) can be uniquely described by a linear combination of the spin-weighted spherical harmonics with spin weight $N$, 
\begin{equation*} {}_NF = \sum_{n=\vert N\vert}^L \sum_{j=-n}^n {}_NF_{n,j} \ {}_NY_{n,j},\end{equation*}
where ${}_NF_{n,j} = \left\langle {}_NF,{}_NY_{n,j}\right\rangle_{\mathrm L^2(\Omega)}$ for the ranges for $n$ and $j$ defined in the sum above. The dimension of $ \mathrm{Harm}^N_{\vert N\vert \dots L}(\Omega)$ is $(L+1)^2-N^2$.
\newline
\newline
To solve the concentration problem in a measurable set $R\subset \Omega$
for bandlimited (by $L$) spin-weighted spherical harmonics with spin
weight $N$, we need to find a function ${}_NF \in
\mathrm{Harm}^N_{\vert N\vert \dots L}(\Omega)$ with a maximum
fraction of its total energy (in the sense of the $\mathrm L^2$-norm) within the region $R$, as is formulated in Problem~\ref{probscar1}.

\begin{prob}[Spin-Weighted Concentration Problem]\label{probscar1}
 For $L \in \mathbb N_0$ and\ a measurable set\ $R\subset \Omega$, find ${}_NF \in
\mathrm{Harm}^N_{\vert N\vert \dots L}(\Omega)$ such that 
\begin{equation}\lambda = \frac{\int_R {}_NF(\xi)\ \overline{{}_NF(\xi)} \ \mathrm d \omega(\xi)}{\int_\Omega {}_NF(\xi)\ \overline{{}_NF(\xi)} \ \mathrm d\omega(\xi)}\label{concprob}\end{equation}
 is maximized. 
\end{prob}

\noindent With
\begin{equation*} \int_R {}_NF(\xi)\ \overline{{}_NF(\xi)} \ \mathrm d\omega(\xi) = \sum_{n=\vert N\vert}^L\sum_{j=-n}^n \sum_{n'=\vert N\vert}^L\sum_{j'=-n'}^{n'} {}_NF_{n,j}\ \overline{{}_NF_{n',j'}} \underbrace{\int_R {}_NY_{n,j}(\xi) \ \overline{{}_NY_{n',j'}(\xi)} \ \mathrm d\omega(\xi)}_{=: \overline{K^N_{nj,n'j'}}=K^N_{n'j',nj}}\end{equation*}
and
\begin{equation*} \int_\Omega {}_NF(\xi)\ \overline{{}_NF(\xi)} \ \mathrm d\omega(\xi) = \sum_{n=\vert N\vert}^L \sum_{j=-n}^n \left\vert {}_NF_{n,j} \right\vert^2,\end{equation*}
we obtain the formulation
\begin{equation*}\lambda = \frac{\overline{G^N}^\mathrm T K^NG^N}{\overline{G^N}^\mathrm T G^N},\end{equation*}
where $G^N:=({}_NF_{00}, \dots ,{}_NF_{LL})^\mathrm T$ and
$K^N:=
\begin{pmatrix}
K^N_{00,00} & \dots & K^N_{00,LL}\\
\vdots & \ddots & \vdots\\
K^N_{LL,00} & \dots & K^N_{LL,LL}
\end{pmatrix}.$
\newline
\newline
Problem~\ref{probscar1} is therefore equivalent to the following problem.
\begin{prob}[Matrix Formulation of the Spin-Weighted Concentration Problem]\label{probscar2}
 Find the eigenvectors $G^N \in \mathbb C^{(L+1)^2-N^2}$ for the eigenvalue problem 
\begin{equation}K^NG^N=\lambda G^N\label{ep_sc}\end{equation}
for which the eigenvalue $\lambda \in \mathbb R$ is maximized.
\end{prob}
\begin{lemma}\label{lemmareal}
The kernel matrix $\overline{K^N}$ is, as a Gramian matrix\ based on linearly independent functions, Hermitian and positive definite. Its eigenvalues are hence real and positive and its eigenvalues form an orthonormal basis. This also holds true for its complex conjugate $K^N$. From equation (\ref{concprob}) follows that $0 \leq \lambda \leq 1$.
\end{lemma}

\begin{theorem}\label{KN}
 The matrix elements of the kernel matrix $K^N$ are given by 
\begin{align*}
K^N_{nj,n'j'}
&=(-1)^{N+j} \sum_{k=\vert n-n'\vert}^{n+n'} \sqrt{\frac{(2n+1)(2n'+1)(2k+1)}{4\pi}} \begin{pmatrix} n& k& n'\\ j& j'-j& -j' \end{pmatrix}\\
&\qquad \times\begin{pmatrix} n& k& n'\\ -N& 0& N \end{pmatrix} \int_R X_{k,j-j'}(t)e^{-i(j-j')\varphi} \ \mathrm d\omega(\xi(t,\varphi))
\end{align*}
for all $N\in\mathbb Z$, all $n,n'\in\mathbb N_0$, $n,n'\geq\vert N\vert$, all $j=-n,\dots, n$, and all $j'=-n',\dots,n'$, where $\begin{pmatrix} j_1& j_2& j\\ m_1& m_2& m \end{pmatrix}$ denotes the Wigner 3j-symbol (see \cite{diss, wigner}).
\end{theorem}

\noindent The proof to Theorem~\ref{KN} is given in \cite{diss}.
\newline
\newline
\noindent The matrix in Problem~\ref{probscar2} is Hermitian, hence the errors in the eigenvalues are limited by numerical errors arising from calculating the regional integrals in Theorem~\ref{KN} (see \cite{bauerfike}). Numerical experiments show that the matrix $K^N$ has eigenvalues close to $1$ as well as close to $0$. This supports the intention behind the definition of Slepian functions: we find well-concentrated functions and can distinguish them from badly concentrated functions. On the other hand, this distribution of the eigenvalues implies a very high condition number of $K^N$. It is known that such a discrepancy between large and small eigenvalues has the following effect: perturbations of the matrix components (which are relevant in our case, if the integrals over $R$ are calculated numerically) can cause a large relative error for the small eigenvalues, which is certainly also connected to the error for the associated eigenvectors, see e.g.~\cite[Section 5.7]{schwarz}. Moreover, high bandlimits, that is large matrices in the eigenvalue problem, are often associated to high numerical costs. Previous experiments with known commuting operators showed that the alternative eigenvalue problem is notably faster to solve.
\newline
\newline
To construct the commuting matrix for the polar cap regions, we first need to derive yet another equivalent problem to Problems~\ref{probscar1} and~\ref{probscar2}. The following derivation holds for all regions~$R$. 
Upon multiplying~(\ref{ep_sc}) by ${}_NY_{n,j}(\eta)$, $\eta\in\Omega$, and summing over all $n=\vert N\vert, \dots, L$ and $j=-n,\dots,n$, we obtain 
\begin{equation*} \sum_{n=\vert N\vert}^L \sum_{j=-n}^n {}_NY_{n,j}(\eta) \sum_{n'=\vert N\vert}^L \sum_{j'=-n'}^{n'} \int_R \overline{{}_NY_{n,j}(\xi)} \ {}_NY_{n',j'}(\xi) \ \mathrm d\omega(\xi) \ {}_NF_{n',j'} = \lambda \sum_{n=\vert N\vert}^L\sum_{j=-n}^n {}_NY_{n,j}(\eta) \ {}_NF_{n,j}. \end{equation*}
By interchanging summation and integration, this leads to
\begin{equation*} \int_R \sum_{n=\vert N\vert}^L\sum_{j=-n}^n \overline{{}_NY_{n,j}(\xi)}\ {}_NY_{n,j}(\eta) \sum_{n'=\vert N\vert}^L \sum_{j'=-n'}^{n'} {}_NF_{n',j'}\ {}_NY_{n',j'}(\xi) \ \mathrm d\omega(\xi) = \lambda \sum_{n=\vert N\vert}^L\sum_{j=-n}^n {}_NF_{n,j}\ {}_NY_{n,j}(\eta).\end{equation*}

\noindent With these considerations, we can reformulate the eigenvalue problem of Problem \ref{probscar2}.

\begin{prob}\label{probscar3}
The eigenvalue problem (\ref{ep_sc}) is equivalent to a homogeneous integral equation of the second kind with a finite-rank, symmetric, and Hermitian kernel, this means that
\begin{equation*}\int_R \mathcal K^N(\xi,\eta) \ {}_NF(\xi) \ \mathrm d\omega(\xi) = \lambda \ {}_NF(\eta),\end{equation*}
where ${}_NF\in\mathrm{Harm}^N_{\vert N\vert\dots L}(\Omega)$ and
\begin{equation*} \mathcal K^N(\xi,\eta):=\sum_{n=\vert N\vert}^L\sum_{j=-n}^n \overline{{}_NY_{n,j}(\xi)}\ {}_NY_{n,j}(\eta), \qquad \xi,\eta\in\Omega.\end{equation*}
\end{prob}

\subsection{Commuting Operator for Spherical Cap Regions}\label{chapcomcap}

\noindent We now turn our attention to the special case of spherical cap regions $R=\lbrace \xi\in\Omega \ \vert\ \xi\cdot\eta \geq b\rbrace$ with $b=\cos \theta \in (0,1)$ and $\eta\in\Omega$. Due to the rotational symmetry of the sphere, it suffices to consider the case of a polar cap, where $\eta = (0,0,1)^\mathrm T$, as Slepian functions constructed for polar caps can be rotated to generic spherical caps by a linear transformation of the coefficients. For polar cap regions, we can write
\begin{equation*} K^N_{nj,n'j'}=\int_R \overline{{}_N Y_{n,j}(\xi)} \ {}_N Y_{n',j'}(\xi) \ \mathrm d\omega(\xi)=\int_0^{2\pi} \int_b^1 \overline{{}_N Y_{n,j}\left(\xi(t,\varphi)\right)} \ {}_N Y_{n',j'}\left(\xi(t,\varphi)\right) \ \mathrm dt \ \mathrm d\varphi \end{equation*}
for all $N\in\mathbb Z,$ all $n,n'=\vert N\vert,\dots,L$, all $j=-n,\dots,n$, and all $j'=-n',\dots,n'$.
\newline
\begin{theorem}\label{KN_cap}
The kernel matrix for the spherical cap can be calculated by
\begin{align*}
K^N_{nj,n'j'}&= \frac{(-1)^{N+j}}{2} \sqrt{(2n+1)(2n'+1)} \sum_{k=\vert n-n'\vert}^{n+n'} \begin{pmatrix} n& k& n'\\ j& 0& -j \end{pmatrix} \begin{pmatrix} n& k& n'\\ -N& 0& N \end{pmatrix}\\
&\qquad \times \left\lbrack P_{k-1}(b)-P_{k+1}(b)\right\rbrack \delta_{j,j'}
\end{align*}
for all $N\in\mathbb Z$, all $n,n'=\vert N\vert,\dots,L$, all $j=-n,\dots,n$, and all $j'=-n',\dots,n'$.
\end{theorem}

\noindent The proof for Theorem~\ref{KN_cap} is given in \cite{diss}.
\newline
\newline
\noindent Next, we define and prove the commuting operator.

\begin{theorem}\label{commute}
For a polar cap with $b=\cos \theta \leq t\leq 1$, the kernel function defined by
\begin{equation*}\mathcal K^N (\xi,\eta):=\sum_{n=\vert N \vert}^L \sum_{j=-n}^n \overline{{}_N Y_{n,j}(\xi)} \ {}_N Y_{n,j}(\eta)\end{equation*}
commutes with the differential operator
\begin{equation*}\mathcal I^N_\xi := \left(b-t_1\right) \Delta^{\ast,N}_\xi + \left(t_1^2-1\right)\partial_{t_1}-L(L+2)t_1\end{equation*}
\noindent for all $N\in\mathbb Z$, where $\Delta^{\ast,N}_\xi$ is defined in Theorem \ref{dglfol}, $L$ is the bandlimit, and $\xi=\xi(t_1,\varphi_1),\eta=\eta(t_2,\varphi_2)\in\Omega$. This means that for any function $u\in\mathrm X^2(\Omega_0)$, we obtain
\begin{align*}
\int_R \overline{\mathcal K^N (\xi,\eta)} \left\lbrack\mathcal I^N_\eta u(\eta)\right\rbrack \ \mathrm d \omega(\eta) &= \int_R \left\lbrack \mathcal I^N_\xi \overline{\mathcal K^N (\xi,\eta)}\right\rbrack u(\eta) \ \mathrm d \omega(\eta)\\
&= \mathcal I^N_\xi \int_R \overline{\mathcal K^N (\xi,\eta)} u(\eta) \ \mathrm d \omega(\eta).
\end{align*}
\end{theorem}

\begin{remark}\label{proofcommute}
Let $N\in\mathbb Z$ and $\xi,\eta\in\Omega$. To prove Theorem \ref{commute}, we have to show that the following holds:
\begin{enumerate}
\item For two functions $u_1$, $u_2\in \mathrm X^2(\Omega_0)$, the differential operator is self-adjoint, that is
\begin{equation*} \int_R \overline{u_1(\xi)} \left\lbrack \mathcal I^N_\xi u_2(\xi)\right\rbrack \ \mathrm d \omega(\xi) = \int_R \overline{\left\lbrack \mathcal I^N_\xi u_1(\xi)\right\rbrack} u_2(\xi) \ \mathrm d \omega(\xi).\end{equation*}
\item\quad $\overline{\mathcal I^N_\xi} \mathcal K^N(\xi,\eta) = \mathcal I^N_\eta \mathcal K^N(\xi,\eta).$
\item \quad $\overline{\overline{\mathcal I^N_\xi} \mathcal K^N(\xi,\eta)} = \mathcal I^N_\xi \ \overline{\mathcal K^N(\xi,\eta)}.$
\item We can interchange the integration and the differential operator, this means that for any $u\in\mathrm X^2(\Omega_0)$
\begin{equation*}\mathcal I^N_\xi \int_R \overline{\mathcal K^N(\xi,\eta)}u(\eta)\ \mathrm d\omega(\eta) = \int_R \mathcal I^N_\xi \overline{\mathcal K^N(\xi,\eta)}u(\eta) \ \mathrm d\omega(\eta).\end{equation*}
\end{enumerate}
\end{remark}

\noindent Theorem~\ref{commute} follows from Remark~\ref{proofcommute} and Corollary~\ref{folYinspace}, because 
\begin{align*}
\int_R \overline{\mathcal K^N (\xi,\eta)} \left\lbrack\mathcal I^N_\eta u(\eta)\right\rbrack \ \mathrm d \omega(\eta) &\overset{1.}{=} \int_R \overline{\left\lbrack \mathcal I^N_\eta \mathcal K^N (\xi,\eta)\right\rbrack} u(\eta) \ \mathrm d \omega(\eta)\\
&\overset{2.}{=} \int_R \overline{\left\lbrack \overline{\mathcal I^N_\xi} \mathcal K^N (\xi,\eta)\right\rbrack} u(\eta) \ \mathrm d \omega(\eta)\\
&\overset{3.}{=} \int_R \mathcal I^N_\xi \ \overline{\mathcal K^N(\xi,\eta)} u(\eta) \ \mathrm d\omega(\eta)\\
&\overset{4.}{=} \mathcal I^N_\xi \int_R \overline{\mathcal K^N(\xi,\eta)} u(\eta) \ \mathrm d\omega(\eta).
\end{align*}

\noindent To prove Theorem~\ref{commute}, we prove Remark~\ref{proofcommute} in the following. 
\begin{proof}
Let $N\in\mathbb Z$ be a given spin weight.
\begin{enumerate}
\item Let $\xi=\xi(t,\varphi)\in\Omega$ and $u_1,u_2\in \mathrm X^2(\Omega_0)$. Then, the left-hand side of the first condition in Remark~\ref{proofcommute} can be reformulated to 
\begin{align}
\int_R \overline{u_1(\xi)} \left\lbrack \mathcal I^N_\xi u_2(\xi)\right\rbrack \ \mathrm d\omega(\xi)&= \int_R (b-t)\overline{u_1(\xi)} \Delta^{\ast,N}_\xi u_2(\xi) \ \mathrm d\omega(\xi) + \int_R \left(t^2-1\right)\overline{u_1(\xi)} \partial_t u_2(\xi) \ \mathrm d\omega(\xi)\nonumber\\
&\quad -\int_R L(L+2)t \overline{u_1(\xi)} u_2(\xi) \ \mathrm d\omega(\xi). \label{inthelp}
\end{align}
These integrals exist, because $u_1,u_2\in\mathrm
X^2(\Omega_0)$ yields that $u_1$, $u_2$,
$\left(t^2-1\right)\partial_t u_2$, and $\Delta^{\ast,N} u_2$ are
 all bounded on $\Omega$. As a consequence, 
$\mathcal I^N u_2$ is also bounded on $\Omega$. We
hence integrate over products of bounded functions, which
themselves are bounded. The same holds true for the right-hand
side of the first condition from Remark~\ref{proofcommute} with
interchanged roles of $u_1$ and $u_2$. Note that in Theorem
\ref{commute}, we use $\mathcal K^N$ instead of $u_1$. 
Corollary~\ref{folYinspace} states that the spin-weighted
spherical harmonics are in $X^2(\Omega_0)$. Therefore, $\mathcal K^N$
consists of products and sums of bounded functions, which are hence
bounded themselves.
\newline
\newline
The first integral on the right-hand side of Equation~(\ref{inthelp}),
 together with Green's second surface identity for the
spin-weighted Beltrami operator for the spherical cap
(Theorem~\ref{green}) leads us to
\begin{align*}
&\int_R (b-t)\overline{u_1(\xi)} \Delta^{\ast,N}_\xi u_2(\xi) \ \mathrm d\omega(\xi)\\
&= \int_R \overline{\Delta^{\ast,N}_\xi} \left((b-t)\overline{u_1(\xi)}\right) u_2(\xi) \ \mathrm d\omega(\xi)\\
&\quad-\int_0^{2\pi} \left\lbrack \left(1-t^2\right)\left((b-t)\overline{u_1(\xi)}\partial_t u_2(\xi)-u_2(\xi)\partial_t\left((b-t)\overline{u_1(\xi)}\right)\right)\right\rbrack_{t=b} \ \mathrm d\varphi.
\end{align*}
Because
\begin{equation*} \overline{\Delta^{\ast,N}_\xi} = \Delta^\ast_\xi -\frac{N^2+2iNt\partial_\varphi}{1-t^2} \end{equation*}
 (see Theorem \ref{dglfol}) and $b-t$ is independent of $\varphi$, we get, consequently,
\begin{align*}
&\Delta^\ast_\xi \left((b-t)\overline{u_1(\xi)}\right)\\
&= \partial_t\left(\left(1-t^2\right)\partial_t (b-t)\overline{u_1(\xi)}\right)+ (b-t)\frac{1}{1-t^2} \ \partial_\varphi^2 \overline{u_1(\xi)}\\
&= \partial_t\left(\left(1-t^2\right)\left(-\overline{u_1(\xi)}+(b-t)\partial_t\overline{u_1(\xi)}\right)\right)+ (b-t)\frac{1}{1-t^2} \ \partial_\varphi^2 \overline{u_1(\xi)}\\
&= \partial_t \left(\left(t^2-1\right)\overline{u_1(\xi)}\right)-\left(1-t^2\right)\partial_t \overline{u_1(\xi)}+ (b-t)\partial_t\left(\left(1-t^2\right)\partial_t\overline{u_1(\xi)}\right)\\
&\quad + (b-t)\frac{1}{1-t^2}\ \partial_\varphi^2 \overline{u_1(\xi)}\\
&= (b-t)\Delta^\ast_\xi \overline{u_1(\xi)} +\partial_t \left(\left(t^2-1\right)\overline{u_1(\xi)}\right) +\left(t^2-1\right)\partial_t \overline{u_1(\xi)},
\end{align*}
 which yields 
\begin{equation*} \overline{\Delta^{\ast,N}_\xi} \left((b-t)\overline{u_1(\xi)}\right) = (b-t)\overline{\Delta^{\ast,N}_\xi u_1(\xi)} +\partial_t \left(\left(t^2-1\right)\overline{u_1(\xi)}\right)+\left(t^2-1\right)\partial_t \overline{u_1(\xi)}. \end{equation*}
Furthermore, we obtain
\begin{align*}
&\left\lbrack \left(1-t^2\right)\left((b-t)\overline{u_1(\xi)}\partial_t u_2(\xi)-u_2(\xi)\partial_t\left((b-t)\overline{u_1(\xi)}\right)\right)\right\rbrack_{t=b}\\
&=\underbrace{\left\lbrack \left(1-t^2\right)(b-t)\overline{u_1(\xi)}\partial_t u_2(\xi)\right\rbrack_{t=b}}_{=0}-\underbrace{\left\lbrack \left(1-t^2\right)(b-t)u_2(\xi)\partial_t\overline{u_1(\xi)}\right\rbrack_{t=b}}_{=0}\\
&\quad + \left\lbrack \left(1-t^2\right)\overline{u_1(\xi)}u_2(\xi)\right\rbrack_{t=b}\\
&= \left\lbrack \left(1-t^2\right)\overline{u_1(\xi)}u_2(\xi)\right\rbrack_{t=b}.
\end{align*}
Altogether, the first integral on the right-hand side of Equation~(\ref{inthelp}) can be expressed as 
\begin{align*}
&\int_R (b-t)\overline{u_1(\xi)} \Delta^{\ast,N}_\xi u_2(\xi) \ \mathrm d\omega(\xi)\\
&=\int_R (b-t)u_2(\xi)\overline{\Delta^{\ast,N}_\xi u_1(\xi)} \ \mathrm d\omega(\xi) +\int_R u_2(\xi)\partial_t \left(\left(t^2-1\right)\overline{u_1(\xi)}\right)\ \mathrm d\omega(\xi)\\
&\quad +\int_R\left(t^2-1\right)u_2(\xi)\partial_t \overline{u_1(\xi)}\ \mathrm d\omega(\xi) - \int_0^{2\pi} \left\lbrack \left(1-t^2\right)\overline{u_1(\xi)}u_2(\xi)\right\rbrack_{t=b} \ \mathrm d\varphi,
\end{align*}
while integration by parts applied to the second integral on the right-hand side of Equation~(\ref{inthelp}) leads to
\begin{align*}
&\int_R \left(t^2-1\right)\overline{u_1(\xi)} \partial_t u_2(\xi) \ \mathrm d\omega(\xi)\\
&= \int_0^{2\pi} \int_b^1 \left(t^2-1\right)\overline{u_1(\xi)} \partial_t u_2(\xi) \ \mathrm dt \ \mathrm d\varphi \\
&= \int_0^{2\pi} \left\lbrack \left(t^2-1\right)\overline{u_1(\xi)} u_2(\xi)\right\rbrack_{t=b}^{t=1} \ \mathrm d\varphi -\int_0^{2\pi} \int_b^1 \partial_t\left(\left(t^2-1\right)\overline{u_1(\xi)}\right) u_2(\xi) \ \mathrm dt \ \mathrm d\varphi\\
&= \int_0^{2\pi} \left\lbrack \left(1-t^2\right)\overline{u_1(\xi)} u_2(\xi)\right\rbrack_{t=b} \ \mathrm d\varphi -\int_0^{2\pi} \int_b^1 \partial_t\left(\left(t^2-1\right)\overline{u_1(\xi)}\right) u_2(\xi) \ \mathrm dt \ \mathrm d\varphi.
\end{align*}
Altogether, we obtain
\begin{align*}
&\int_R \overline{u_1(\xi)} \left\lbrack \mathcal I^N_\xi u_2(\xi)\right\rbrack \ \mathrm d\omega(\xi)\\
&= \int_R (b-t)\overline{\Delta^{\ast,N}_\xi u_1(\xi)} u_2(\xi) \ \mathrm d\omega(\xi) + \int_R \left(t^2-1\right)\left(\partial_t \overline{u_1(\xi)}\right) u_2(\xi) \ \mathrm d\omega(\xi)\\
&\quad - \int_R L(L+2)t \overline{u_1(\xi)} u_2(\xi) \ \mathrm d\omega(\xi)\\
&= \int_R \overline{\left\lbrack \mathcal I^N_\xi u_1(\xi)\right\rbrack} u_2(\xi) \ \mathrm d\omega(\xi).
\end{align*}

\item To prove the second equation in Remark~\ref{proofcommute} for $\xi=\xi(t_1,\varphi_1)\in \Omega$ and $\eta=\eta(t_2,\varphi_2)\in\Omega$, we invoke Theorems~\ref{dglfol} and~\ref{recspin},
\allowdisplaybreaks
\begin{align*}
&\left(\overline{\mathcal I^N_\xi} -\mathcal I^N_\eta\right) \mathcal K^N(\xi,\eta)\\
&= \left(\left(b-t_1\right)\overline{\Delta^{\ast,N}_\xi} + \left(t_1^2-1\right)\partial_{t_1}-L(L+2)t_1-\left(b-t_2\right)\Delta^{\ast,N}_\eta - \left(t_2^2-1\right)\partial_{t_2} \right.\\
&\quad \Bigl. +L(L+2)t_2\Bigr)\sum_{n=\vert N \vert}^L \sum_{j=-n}^n \overline{{}_N Y_{n,j}(\xi)} \ {}_NY_{n,j}(\eta)\\
&= \left(t_1-t_2\right) \sum_{n=\vert N\vert}^L \sum_{j=-n}^n (n(n+1)-L(L+2)) \overline{{}_N Y_{n,j}(\xi)} \ {}_NY_{n,j}(\eta) \\
&\quad+ \sum_{n=\vert N\vert}^L \sum_{j=-n}^n \left(t_1^2-1\right) \partial_{t_1} \ \overline{{}_N Y_{n,j}(\xi)} \ {}_NY_{n,j}(\eta) - \sum_{n=\vert N\vert}^L \sum_{j=-n}^n \left(t_2^2-1\right) \overline{{}_N Y_{n,j}(\xi)} \partial_{t_2} \ {}_NY_{n,j}(\eta)\\
&\overset{(\ref{rec1})}{=} \left(t_1-t_2\right) \sum_{n=\vert N\vert}^L \sum_{j=-n}^n (n(n+1)-L(L+2)) \overline{{}_N Y_{n,j}(\xi)} \ {}_NY_{n,j}(\eta) \\
&\quad + \sum_{n=\vert N\vert}^L \sum_{j=-n}^n \biggl(\left( nt_1 +\frac{Nj}{n}\right) \overline{{}_N Y_{n,j}(\xi)} \ {}_NY_{n,j}(\eta)-(2n+1)\alpha^N_{n,j} \ \overline{{}_N Y_{n-1,j}(\xi)} \ {}_NY_{n,j}(\eta)\biggr)\\
&\quad - \sum_{n=\vert N\vert}^L \sum_{j=-n}^n \biggl(\left( nt_2 +\frac{Nj}{n}\right) \overline{{}_N Y_{n,j}(\xi)} \ {}_NY_{n,j}(\eta) -(2n+1)\alpha^N_{n,j} \ \overline{{}_N Y_{n,j}(\xi)} \ {}_NY_{n-1,j}(\eta)\biggr)\\
&= \left(t_1-t_2\right) \sum_{n=\vert N\vert}^L \sum_{j=-n}^n (n(n+2)-L(L+2)) \overline{{}_N Y_{n,j}(\xi)} \ {}_NY_{n,j}(\eta) \\
&\quad + \sum_{n=\vert N\vert}^L \sum_{j=-n}^n (2n+1)\alpha^N_{n,j} \left( \overline{{}_N Y_{n,j}(\xi)} \ {}_NY_{n-1,j}(\eta)-\overline{{}_N Y_{n-1,j}(\xi)} \ {}_NY_{n,j}(\eta)\right).
\end{align*}
Note that we used $\overline{\Delta^{\ast,N}_\xi} \ \overline{{}_NY_{n,j}(\xi)} = \overline{\Delta^{\ast,N}_\xi \ {}_NY_{n,j}(\xi)}$.
\newline
\newline
Furthermore, Theorem~\ref{CDF} together with the relationship 
\begin{equation*} \sum_{n=\vert N\vert}^L\sum_{j=-n}^n\sum_{k=n_j}^{n-1} b_{n,j,k} = \sum_{k=\vert N\vert}^{L-1}\sum_{j=-k}^k\sum_{n=k+1}^L b_{n,j,k},\end{equation*}
yields
\begin{align*}
&\left(\overline{\mathcal I^N_\xi} -\mathcal I^N_\eta\right) \mathcal K^N(\xi,\eta)\\
&= \left(t_1-t_2\right) \sum_{n=\vert N\vert}^L \sum_{j=-n}^n (n(n+2)-L(L+2)) \overline{{}_N Y_{n,j}(\xi)} \ {}_NY_{n,j}(\eta) \\
&\quad + \left(t_1-t_2\right) \sum_{n=\vert N\vert}^L \sum_{j=-n}^n (2n+1) \sum_{k=n_j}^{n-1} \overline{{}_N Y_{k,j}(\xi)} \ {}_NY_{k,j}(\eta)\\
&= \left(t_1-t_2\right) \sum_{n=\vert N\vert}^L \sum_{j=-n}^n (n(n+2)-L(L+2)) \overline{{}_N Y_{n,j}(\xi)} \ {}_NY_{n,j}(\eta) \\
&\quad + \left(t_1-t_2\right) \sum_{k=\vert N\vert}^{L-1} \sum_{j=-k}^k \overline{{}_N Y_{k,j}(\xi)} \ {}_NY_{k,j}(\eta) \sum_{n=k+1}^L (2n+1)\\
&= \left(t_1-t_2\right) \sum_{n=\vert N\vert}^{L-1} \sum_{j=-n}^n \overline{{}_N Y_{n,j}(\xi)} \ {}_NY_{n,j}(\eta) \biggl\lbrack n(n+2)-L(L+2)\\
 &\quad + \sum_{k=n+1}^L (2k+1) \biggr\rbrack\\
&=0,
\end{align*}
because $ n(n+2)-L(L+2) + \sum_{k=n+1}^L (2k+1) = 0.$ This is equivalent to the proposition
\begin{equation*} \overline{\mathcal I^N_\xi} \mathcal K^N(\xi,\eta) = \mathcal I^N_\eta \mathcal K^N(\xi,\eta).\end{equation*}

\item To prove the third equality of Remark~\ref{proofcommute}, we use that $\mathcal K^N$ is self-adjoint 
\begin{equation*}\overline{\mathcal K^N(\xi,\eta)}=\mathcal K^N(\eta,\xi).\end{equation*}
Using the definition of $\Delta^{\ast,N}$ given in Theorem~\ref{dglfol}, we obtain
\begin{equation*}\mathcal I^N_\xi = \partial_t\left((b-t)\left(1-t^2\right)\partial_t\right)-\left(\frac{N^2(b-t)}{1-t^2}+L(L+2)t\right) +\frac{b-t}{1-t^2}\left(\partial_\varphi^2+2iNt\partial_\varphi\right).\end{equation*}
We can now utilize Theorem~\ref{wigner} to derive the identity
\begin{align*}
&\mathcal I^N_\xi \ \overline{\mathcal K^N(\xi,\eta)}\\
&= \sum_{n=\vert N\vert}^L\sum_{j=-n}^n \overline{{}_NY_{n,j}(\eta)} \ \mathcal I^N_\xi \ {}_NY_{n,j}(\xi)\\
&=\sum_{n=\vert N\vert}^L\sum_{j=-n}^n \overline{{}_NY_{n,j}(\eta)}\biggl( \partial_t\left((b-t)\left(1-t^2\right)\partial_t\ {}_NY_{n,j}(\xi)\right)\\
&\quad -\left(\frac{N^2(b-t)}{1-t^2}+L(L+2)t\right)\ {}_NY_{n,j}(\xi)+\frac{b-t}{1-t^2}\left(-j^2-2jNt\right)\ {}_NY_{n,j}(\xi)\biggr).
\end{align*}
\noindent Analogously, the left-hand side of the third condition in Remark~\ref{proofcommute} yields 
\begin{align*}
&\overline{\mathcal I^N_\xi} \mathcal K^N(\xi,\eta)\\
&=\sum_{n=\vert N\vert}^L\sum_{j=-n}^n {}_NY_{n,j}(\eta) \ \overline{\mathcal I^N_\xi} \ \overline{{}_NY_{n,j}(\xi)}\\
&= \sum_{n=\vert N\vert}^L\sum_{j=-n}^n {}_NY_{n,j}(\eta)\biggl( \partial_t\left((b-t)\left(1-t^2\right)\partial_t\ \overline{{}_NY_{n,j}(\xi)}\right)\\
&\quad -\left(\frac{N^2(b-t)}{1-t^2}+L(L+2)t\right)\ \overline{{}_NY_{n,j}(\xi)}+\frac{b-t}{1-t^2}\left(-j^2-2jNt\right)\ \overline{{}_NY_{n,j}(\xi)}\biggr).
\end{align*}
As a consequence, we get
\begin{align*}
&\overline{\overline{\mathcal I^N_\xi} \mathcal K^N(\xi,\eta)}\\
&=\sum_{n=\vert N\vert}^L\sum_{j=-n}^n \overline{{}_NY_{n,j}(\eta)}\biggl( \partial_t\left((b-t)\left(1-t^2\right)\partial_t\ {}_NY_{n,j}(\xi)\right)\\
&\quad -\left(\frac{N^2(b-t)}{1-t^2}+L(L+2)t\right)\ {}_NY_{n,j}(\xi)+\frac{b-t}{1-t^2}\left(-j^2-2jNt\right)\ {}_NY_{n,j}(\xi)\biggr).
\end{align*}
Hence, the left- and the right-hand side of the third condition in Remark \ref{proofcommute} are equal.

\item Since the functions occurring in the fourth condition of Remark \ref{proofcommute} are sufficiently smooth, it is easy to verify that the conditions for interchanging differentiation and integration are satisfied.
\end{enumerate}
\end{proof}

\begin{theorem}\label{commutesphere}
The commuting relation also holds true for an integral over the unit sphere. For $\xi\in\Omega$,
\begin{align*}
\int_\Omega \overline{\mathcal K^N (\xi,\eta)} \left\lbrack\mathcal I^N_\eta u(\eta)\right\rbrack \ \mathrm d \omega(\eta) &= \int_\Omega \left\lbrack \mathcal I^N_\xi \overline{\mathcal K^N (\xi,\eta)}\right\rbrack u(\eta) \ \mathrm d \omega(\eta)\\
&= \mathcal I^N_\xi \int_\Omega \overline{\mathcal K^N (\xi,\eta)} u(\eta) \ \mathrm d \omega(\eta).
\end{align*}
\end{theorem}

\begin{proof}
Note that we did not have any constraints on $R$ which would have excluded the case $R=\Omega$.
\end{proof}

\begin{fol}\label{Cor:I_ist_endomorph}
The operator $\mathcal{I}^N$ is an endomorphism on every $\mathrm{Harm}^N_{|N|\dots L}(\Omega)$, that is $\mathcal{I}^N Y\in\mathrm{Harm}^N_{|N|\dots L}(\Omega)$ for every $Y\in\mathrm{Harm}^N_{|N|\dots L}(\Omega)$.
\end{fol}
\begin{proof}
Clearly, $\mathcal{K}^N$ from Problem \ref{probscar3} is the reproducing kernel of $\mathrm{Harm}^N_{|N|\dots L}(\Omega)$ in the sense that
\begin{equation*}
 \int_\Omega Y(\eta) \overline{\mathcal{K}^N(\xi,\eta)}\ \mathrm{d}\omega(\eta) = Y(\xi) 
\end{equation*}
for all $\xi\in\Omega$ and all $Y\in\mathrm{Harm}^N_{|N|\dots L}(\Omega)$. More generally, we have automatically
\begin{equation*}
 \int_\Omega Y(\eta) \overline{\mathcal{K}^N(\xi,\eta)}\ \mathrm{d}\omega(\eta) = \mathcal{P}_{\mathrm{Harm}^N_{|N|\dots L}(\Omega)}Y(\xi)
\end{equation*}
for all $\xi\in\Omega$ and all $Y\in\mathrm{L}^2(\Omega)$, where $\mathcal{P}_{\mathrm{Harm}^N_{|N|\dots L}(\Omega)}$ is the orthogonal projection onto $\mathrm{Harm}^N_{|N|\dots L}(\Omega)$.
Hence, if now $Y\in\mathrm{Harm}^N_{|N|\dots L}(\Omega)$ is arbitrary, then we obtain together with Theorem \ref{commutesphere}
\begin{equation*}
\mathcal{I}_\xi^N Y(\xi) = \mathcal{I}_\xi^N \int_\Omega Y(\eta) 
\overline{\mathcal{K}^N(\xi,\eta)}\ \mathrm{d}\omega(\eta) 
= \int_\Omega \left[\mathcal{I}_\eta^N Y(\eta)\right] \overline{\mathcal{K}^N(\xi,\eta)}\ \mathrm{d}\omega(\eta).
\end{equation*}
Thus, $\mathcal{I}^N Y\in\mathrm{Harm}^N_{|N|\dots L}(\Omega)$.
\end{proof}

\subsection{Computation of Slepian Functions for Spherical Cap Regions}\label{compu}

 For polar cap regions, we can use the commuting operator
$\mathcal I^N$ of Section~\ref{chapcomcap} to construct a commuting
Matrix $I^N$, which, as we will show, has the same eigenvectors as
$K^N$. By solving for the eigenvectors of $I^N$ instead of $K^N$, we
obtain the same spin-weighted Slepian functions but with an increased
numerical stability and at a lower computational cost. See also
\cite{simonspaper} for the scalar case.
\newline
\newline
Previously, we showed that $\mathcal I^N_\xi$ and
$\mathcal K^N$ commute for all $N\in\mathbb Z$. Next, we need to
show that they have the same eigenfunctions $\mathcal
G^N_\alpha$, meaning that for $\xi \in\Omega$ and $N\in\mathbb Z$
\begin{align*}
\mathcal I^N_\xi \mathcal G^N_\alpha(\xi)&= \chi_\alpha \mathcal G^N_\alpha(\xi),\\
\int_R \overline{\mathcal K^N(\xi,\eta)} \mathcal G^N_\alpha(\eta) \ \mathrm d\omega(\eta) &= \lambda_{\alpha} \mathcal G^N_\alpha(\xi),
\end{align*}
where $\lambda_{\alpha}$ and $\chi_\alpha$ are not necessarily equal. 

\begin{remark} 
The matrix problem equivalent to $\mathcal I^N_\xi \mathcal G^N_\alpha(\xi)= \chi_\alpha \mathcal G^N_\alpha(\xi)$ is $I^N G^N_\alpha = \chi_\alpha G^N_\alpha$, 
 where
\begin{equation*}I^N_{nj,n'j'} := \int_\Omega \overline{{}_N Y_{n,j}(\xi)}\left( \mathcal I^N_\xi \ {}_NY_{n',j'}(\xi)\right)\ \mathrm d\omega(\xi),\end{equation*}
\begin{equation*}I^N:=
\begin{pmatrix}
I^N_{\vert N\vert,-\vert N\vert,\vert N\vert,-\vert N\vert} & \dots & I^N_{\vert N\vert,-\vert N\vert,LL}\\
\vdots & \ddots & \vdots\\
I^N_{LL,\vert N\vert,-\vert N\vert} & \dots & I^N_{LL,LL}
\end{pmatrix},\end{equation*}
and $G_\alpha^N = \begin{pmatrix} (G^N_{\lvert N \rvert,-\lvert N
 \rvert})_\alpha,\ldots,(G^N_{L,L})_\alpha \end{pmatrix}^T$, with $
(G^N_{n,j})_\alpha = \left\langle \mathcal G^N_\alpha,{}_NY_{n,j}\right\rangle_{\mathrm L^2(\Omega)}$.
\end{remark}
\begin{proof} The eigenvalue problem
\begin{equation*} \mathcal I^N_\xi \mathcal G^N_\alpha(\xi)= \chi_\alpha \mathcal G^N_\alpha(\xi)\end{equation*}
 is equivalent to 
\begin{equation*} \sum_{n'=\vert N\vert}^L \sum_{j'=-n'}^{n'} \left(G^N_{n',j'}\right)_\alpha \mathcal I^N_\xi \ {}_NY_{n',j'}(\xi)= \chi_\alpha \sum_{n'=\vert N\vert}^L \sum_{j'=-n'}^{n'} \left(G^N_{n',j'}\right)_\alpha \ {}_NY_{n',j'}(\xi). \end{equation*}
Upon multiplying by $\overline{{}_N Y_{n,j}(\xi)}$, $n=\vert N\vert,\dots,L$, $j=-n,\dots,n$, we obtain
\begin{equation*}\sum_{n'=\vert N\vert}^L \sum_{j'=-n'}^{n'} \left(G^N_{n',j'}\right)_\alpha \ \overline{{}_N Y_{n,j}(\xi)}\left( \mathcal I^N_\xi \ {}_NY_{n',j'}(\xi)\right)= \chi_\alpha \sum_{n'=\vert N\vert}^L \sum_{j'=-n'}^{n'} \left(G^N_{n',j'}\right)_\alpha \ \overline{{}_N Y_{n,j}(\xi)} \ {}_NY_{n',j'}(\xi).\end{equation*}
Integration over the unit sphere $\Omega$ and interchanging of sum and
integral, together with Theorem \ref{orth} yields 
\begin{align*}
\sum_{n'=\vert N\vert}^L \sum_{j'=-n'}^{n'} \left(G^N_{n',j'}\right)_\alpha \int_\Omega \overline{{}_N Y_{n,j}(\xi)}&\left( \mathcal I^N_\xi \ {}_NY_{n',j'}(\xi)\right)\ \mathrm d\omega(\xi)\\
= \chi_\alpha \sum_{n'=\vert N\vert}^L \sum_{j'=-n'}^{n'} \left(G^N_{n',j'}\right)_\alpha &\underbrace{\int_\Omega\overline{{}_N Y_{n,j}(\xi)} \ {}_NY_{n',j'}(\xi) \ \mathrm d\omega(\xi)}_{=\delta_{n,n'}\delta_{j,j'}}
\end{align*}
 and hence 
\begin{equation}\label{Eq:orthogR_vm} \sum_{n'=\vert N\vert}^L \sum_{j'=-n'}^{n'} \left(G^N_{n',j'}\right)_\alpha I^N_{nj,n'j'} = \chi_\alpha \left(G^N_{n,j}\right)_\alpha \end{equation}
for all $n=\vert N\vert,\dots,L$ and all $j=-n,\dots,n$.\\
Vice versa, if \eqref{Eq:orthogR_vm} holds true, then the linearity of $\mathcal{I}^N$ yields
\begin{align*}
 \chi_\alpha \mathcal{G}_\alpha^N(\xi) & = \sum_{n=|N|}^L \sum_{j=-n}^n\chi_\alpha \left(G_{n,j}^N\right)_\alpha {}_NY_{n,j}(\xi) \\
 & =\sum_{n=|N|}^L \sum_{j=-n}^n \sum_{n'=\vert N\vert}^L \sum_{j'=-n'}^{n'} \left(G^N_{n',j'}\right)_\alpha I^N_{nj,n'j'}\,{}_NY_{n,j}(\xi) \\
 & =\sum_{n=|N|}^L \sum_{j=-n}^n \sum_{n'=\vert N\vert}^L \sum_{j'=-n'}^{n'}
 \left(G^N_{n',j'}\right)_\alpha \int_\Omega \overline{{}_NY_{n,j}(\eta)} \left(\mathcal{I}_\eta^N {}_NY_{n',j'}(\eta)\right)\ \mathrm{d}\omega(\eta)\ {}_NY_{n,j}(\xi) \\
 & =\sum_{n=|N|}^L \sum_{j=-n}^n \int_\Omega \overline{{}_NY_{n,j}(\eta)} \left[\mathcal{I}_\eta^N \sum_{n'=\vert N\vert}^L \sum_{j'=-n'}^{n'} \left(G^N_{n',j'}\right)_\alpha {}_NY_{n',j'}(\eta)\right]\ \mathrm{d}\omega(\eta)\ {}_NY_{n,j}(\xi).
\end{align*}
Hence, Theorem \ref{only}, Remark \ref{Rem:HarmpqN}, and Corollary \ref{Cor:I_ist_endomorph} imply
\begin{equation*}
 \chi_\alpha \mathcal{G}_\alpha^N(\xi) = \sum_{n=|N|}^L \sum_{j=-n}^n
 \left\langle \mathcal{I}^N \mathcal{G}_\alpha^N,{}_N Y_{n,j} \right\rangle_{\mathrm{L}^2(\Omega)} {}_NY_{n,j}(\xi) = \mathcal{I}_\xi^N \mathcal{G}_\alpha^N(\xi),
\end{equation*}
which completes the proof.
\end{proof}

\begin{theorem}\label{commmat}
$K^N$ and $I^N$ also commute for all $N\in\mathbb Z$, that is
\begin{equation*}K^NI^N=I^NK^N.\end{equation*}
\end{theorem}

\begin{proof}
Let $N\in\mathbb Z$, $n,l=\vert N\vert,\dots,L$, $j=-n,\dots,n$, and
$m=-l,\dots,l$. Then the left-hand side together with Theorem \ref{commutesphere}, Theorem \ref{orth}, and Corollary \ref{folYinspace} leads to 
\begin{align*}
&\left(K^NI^N\right)_{nj,lm}\\
&= \sum_{n'=\vert N\vert}^L \sum_{j'=-n}^n K^N_{nj,n'j'} I^N_{n'j',lm}\\
&=\int_R \overline{{}_N Y_{n,j}(\xi)} \int_\Omega \overline{\mathcal{K}^N(\xi,\eta)} \left(\mathcal{I}_\eta^N {}_N Y_{l,m}(\eta)\right) \ \mathrm{d}\omega(\eta) \ \mathrm{d}\omega(\xi)\\
&= \int_R \overline{{}_N Y_{n,j}(\xi)} \,\mathcal{I}_\xi^N \left[ \int_\Omega \overline{\mathcal{K}^N(\xi,\eta)} {}_NY_{l,m}(\eta)\ \mathrm{d}\omega(\eta) \right] \ \mathrm{d}\omega(\xi) \\
&=\int_R \overline{{}_N Y_{n,j}(\xi)} \left(\mathcal I^N_\xi \ {}_NY_{l,m}(\xi)\right) \ \mathrm d\omega(\xi).
\end{align*}
\noindent The right-hand side in combination with Theorem~\ref{commute} and Corollary \ref{folYinspace} yields the same result
\begin{align*}
&\left(I^NK^N\right)_{nj,lm}\\
&= \sum_{n'=\vert N\vert}^L \sum_{j'=-n}^n I^N_{nj,n'j'} K^N_{n'j',lm}\\
&=\int_\Omega\overline{{}_N Y_{n,j}(\xi)} \int_R \left(\mathcal{I}_\xi^N \overline{\mathcal{K}^N(\xi,\eta)}\right) {}_NY_{l,m}(\eta)\ \mathrm{d}\omega(\eta)\ \mathrm{d}\omega(\xi)
\\
&=\int_\Omega \overline{{}_N Y_{n,j}(\xi)} \int_R \overline{\mathcal{K}^N(\xi,\eta)} \left(\mathcal{I}_\eta^N {}_NY_{l,m}(\eta)\right) \ \mathrm{d}\omega(\eta)\ \mathrm{d}\omega(\xi)
\\
&=\int_R \overline{{}_N Y_{n,j}(\eta)} \left(\mathcal I^N_\eta \ {}_NY_{l,m}(\eta)\right) \ \mathrm d\omega(\eta).
\end{align*}
Note that we used here again that $\mathcal{K}^N$ is the reproducing kernel of $\mathrm{Harm}^N_{|N|\dots L}(\Omega)$ (see also the proof of Corollary \ref{Cor:I_ist_endomorph}).
\end{proof}

\noindent Note that because $I^N,K^N \in \mathbb
R^{\lbrack(L+1)^2-N^2\rbrack \times \lbrack(L+1)^2-N^2\rbrack}$ (as a result of $n=\vert N\vert,\dots,L$, $j=-n,\dots,n$), we
obtain $(L+1)^2-N^2$ orthogonal eigenvectors $G^N_\alpha$, orthogonal
eigenfunctions $\mathcal G^N_\alpha$, and eigenvalues
$\lambda_\alpha$, where $\alpha=1,\dots,(L+1)^2-N^2$.

\begin{lemma}\label{IN}
The components of $I^N$ have the following form:
\begin{align*}
I^N_{nj,nj}&=-\left\lbrack n(n+1)b+Nj\left(1-\frac{L(L+2)+1}{n(n+1)}\right) \right\rbrack,\\
I^N_{nj,n+1,j}&= \left \lbrack (n+1)^2-1-L(L+2)\right\rbrack \alpha^N_{n+1,j}\\
&= \left\lbrack n(n+2)-L(L+2)\right\rbrack \frac{\sqrt{(n+1-N)(n+1+N)}}{n+1} \ \sqrt{\frac{(n+1-j)(n+1+j)}{(2n+1)(2n+3)}},\\
I^N_{n+1,j,nj}&= \left \lbrack n(n+2)-L(L+2)\right\rbrack \alpha^N_{n+1,j} = I^N_{nj,n+1,j},\\
I^N_{nj,n'j'}&= 0, \qquad \text{else}
\end{align*}
for all $N\in\mathbb Z$, all $n,n'=\vert N\vert,\dots,L$, all
$j=-n,\dots,n$, and all $j'=-n',\dots,n$. Therefore, $I^N$ is a symmetric tridiagonal matrix.
\end{lemma}

\begin{proof}
 Using Theorems~\ref{dglfol} and~\ref{recspin}, we can express
 $\mathcal I^N_\xi \ {}_NY_{n,j}(\xi)$ as the following linear
 combination of the $ {}_NY_{n,j}(\xi)$ for all $\xi\in\Omega$. The derivation is easy but slightly lengthy and is, therefore, omitted here. For a step-by-step proof, see \cite{diss}.
\begin{align*}
&\mathcal I^N_\xi \ {}_NY_{n,j}(\xi)\\
&=-\left\lbrack n(n+1)b+\frac{Nj}{n(n+1)}(n(n+1)-1-L(L+2))\right\rbrack \ {}_NY_{n,j}(\xi)\\
&\qquad + \left\lbrack n^2-1-L(L+2)\right\rbrack \alpha^N_{n,j}
 \ {}_NY_{n-1,j}(\xi) \\
&\qquad + \left\lbrack n(n+2)-L(L+2)\right\rbrack \alpha^N_{n+1,j} \ {}_NY_{n+1,j}(\xi).
\end{align*}
Lemma~\ref{IN} follows from the orthonormality of the spin-weighted spherical harmonics (Theorem~\ref{orth}). For example,
\begin{align*}
I_{nj,n+1,j}^N&=\int_\Omega \overline{{}_N Y_{n,j}(\xi)} \left(
\mathcal{I}_\xi^N {}_N Y_{n+1,j}(\xi)\right) \ \mathrm{d}\omega(\xi)\\
 &= \left[(n+1)^2 -1 -L(L+2)\right]\alpha_{n+1,j}^N\int_\Omega \overline{{}_N Y_{n,j}(\xi)}{}_N Y_{n+1-1,j}(\xi)\ \mathrm{d}\omega(\xi). \qedhere
\end{align*}
\end{proof}

\begin{fol}\label{simple}
For each $N\in\mathbb Z$, the commuting matrix $I^N$ is a block-diagonal matrix where each block is a symmetric tridiagonal matrix with nonzero off-diagonal elements. As a consequence, each block has a simple spectrum and the eigenvectors of $I^N$ and $K^N$ coincide.
\end{fol}

\begin{proof}
The block-tridiagonal structure is an immediate consequence of Lemma
\ref{IN}. The simple spectrum follows from basic linear
algebra (see \cite{LA_tridiag}), as does the equivalence of the
eigenvectors from commutation and simple spectrum of one of the
commuting matrix blocks. 
\end{proof}

 \subsection{Properties of the Spin-Weighted Slepian Functions} 

\noindent The spatially concentrated, spectrally limited spin-weighted Slepian
functions for spin weight~$N$, region~$R$, and bandlimit~$L$ are
\begin{equation}\label{spinslep}
\mathcal G^N_\alpha(\xi) = \sum_{n=\lvert N \rvert}^L \sum_{j=-n}^n
\left(G^N_{n,j}\right)_\alpha \ {}_NY_{n,j}(\xi),
\end{equation}
for $\alpha=1,\ldots, (L+1)^2-N^2$ and $\xi \in \Omega$.
 Without loss of generality, we order these Slepian functions
such that their eigenvalues are in descending sequence 
 $1\geq \lambda_1\geq \lambda_2\geq \dots\geq
\lambda_{(L+1)^2-N^2} \geq 0$.
\newline
\newline
The proofs of the following theorems are analogous to the spin-weight-free case ($N=0$) see \cite{dahlen2, diss, simonsdoublecap, simonspaper}.

\begin{theorem}\label{propscar}
The Slepian functions and their corresponding coefficient vectors are orthonormal on the unit sphere and orthogonal within the region of interest $R$
\begin{align}
\sum_{n=\vert N\vert}^L \sum_{j=-n}^n \left(G^N_{n,j}\right)_\alpha \ \overline{\left(G^N_{n,j}\right)_\beta} &= \delta_{\alpha,\beta},\label{g_on}\\
\sum_{n=\vert N\vert}^L \sum_{j=-n}^n \sum_{n'=\vert N\vert}^L \sum_{j'=-n'}^{n'}\left(G^N_{n,j}\right)_\alpha \ \overline{K^N_{nj,n'j'}}\ \overline{\left(G^N_{n',j'}\right)_\beta} &= \lambda_\alpha \delta_{\alpha,\beta},\label{g_og}\\
\left\langle \mathcal G^N_\alpha, \mathcal G^N_\beta\right\rangle_{\mathrm L^2(\Omega)} &= \delta_{\alpha,\beta},\label{gfunc_on}\\
\left\langle \mathcal G^N_\alpha, \mathcal G^N_\beta\right\rangle_{\mathrm L^2(R)} &= \lambda_\alpha \delta_{\alpha,\beta}\label{gfunc_og}
\end{align}
for all $\alpha,\beta=1,\dots,(L+1)^2-N^2$.
\end{theorem}

\begin{theorem}\label{propscar2}
 Each construction of spin-weighted Slepian functions
$\lbrace \mathcal
G^N_\alpha\rbrace_{\alpha=1,\dots,(L+1)^2-N^2}$ for any
region $R$ forms a complete orthonormal basis system of
$\mathrm{Harm}^N_{\vert N\vert \dots L}(\Omega)$. Therefore, any
 ${}_NF\in\mathrm{Harm}^N_{\vert N\vert\dots L}(\Omega)$ 
can be expressed both in the basis of the spin-weighted
spherical harmonics and in the basis of the spin-weighted Slepian
functions
\begin{equation*} {}_NF(\xi)= \sum_{n=\vert N\vert}^L \sum_{j=-n}^n
 \underbrace{\left\langle {}_NF,{}_NY_{n,j}\right\rangle_{\mathrm
 L^2(\Omega)}}_{=:{}_NF_{n,j}} \ {}_NY_{n,j}(\xi)
 =\sum_{\alpha=1}^{(L+1)^2-N^2} \underbrace{\left\langle
 {}_NF,\mathcal G^N_\alpha\right\rangle_{\mathrm
 L^2(\Omega)}}_{=:{}_NF_\alpha} \mathcal
 G^N_\alpha(\xi),\end{equation*}
for $\xi\in\Omega$.
\end{theorem}

\begin{theorem}\label{propscar3}
 The spin-weighted spherical harmonics for degrees $\vert
N\vert$ to $L$ can be expressed in the basis of the
spin-weighted Slepian functions
\begin{equation} {}_NY_{n,j}= \sum_{\alpha=1}^{(L+1)^2-N^2} \overline{\left(G^N_{n,j}\right)_\alpha} \mathcal G^N_\alpha, \label{basistrans} \end{equation}
where
\begin{equation} \sum_{\alpha=1}^{(L+1)^2-N^2} \left(G^N_{n,j}\right)_\alpha \overline{\left(G^N_{n',j'}\right)_\alpha}= \delta_{n,n'}\delta_{j,j'}\label{g_on_basistrans}.\end{equation}
\end{theorem}

\begin{theorem}\label{propscar4}
The spin-weighted Slepian functions also fulfill the following properties
\begin{align}
\sum_{\alpha=1}^{(L+1)^2-N^2} \lambda_\alpha \left(G^N_{n,j}\right)_\alpha \overline{\left(G^N_{n',j'}\right)_\alpha} &= K^N_{nj,n'j'}, \label{on_sca_weighted}\\
\sum_{\alpha=1}^{(L+1)^2-N^2} \lambda_\alpha \mathcal G^N_\alpha(\xi) \overline{ \mathcal G^N_\alpha(\eta)} &= \sum_{n=\vert N\vert}^L\sum_{j=-n}^n \sum_{n'=\vert N\vert}^L\sum_{j'=-n'}^{n'} {}_NY_{n,j}(\xi) K^N_{nj,n'j'}\ \overline{{}_NY_{n',j'}(\eta)}\label{og_sca_weighted}
\end{align}
for all $n,n'=\vert N \vert,\dots,L$, all $j=-n,\dots,n$ ,all $j'=-n',\dots,n'$, and all $\xi,\eta\in\Omega$.
\end{theorem}

\noindent As we mentioned above, eigenvalues of spin-weighted Slepian functions often
cluster around $\lambda^N \approx 1$ and $\lambda^N \approx 0$. The
eigenvalue number at which this transition takes place can be
predicted by the Shannon number, which we derive for the spin-weighted
Slepian functions in the next section.

\subsection{Shannon Number}\label{scaslepshannon}

If the eigenvalues of a matrix $K^N$ have a bimodal distribution with
clusters at 1 and 0, then the Shannon number
\begin{equation*} S^N=\sum_{\alpha =1}^{(L+1)^2-N^2} \lambda_\alpha
 = \mathrm{tr}\left(K^N\right),\end{equation*} predicts the
number of eigenvalues close to 1. Therefore, $S^N$ 
provides an estimation of the dimension of the space of signals
of spin weight $N\in\mathbb Z$ that are both bandlimited by $L$
 and optimally concentrated in $R$. The basis of this space
is given by the eigenfunctions $\mathcal G^N_1,\mathcal
G^N_2,\dots,\mathcal G^N_{S^N}$.

\begin{lemma}
 The Shannon number $S^N$ of spin-weighted Slepian functions and
hence the trace of the matrix $K^N$ only depends on the bandwidth~$L$,
the spin weight~$N$, and the area $A$ of the region $R$ on the unit
sphere 
\begin{equation*} S^N=\left((L+1)^2-N^2\right) \ \frac{A}{4\pi}. \end{equation*}
\end{lemma}

\begin{proof}
Corollary~\ref{addtheo} from the addition theorem yields 
\begin{align*}
S^N&= \sum_{n=\vert N\vert}^L \sum_{j=-n}^n K^N_{nj,nj}\\
&=\int_R \sum_{n=\vert N\vert}^L \sum_{j=-n}^n \overline{{}_NY_{n,j}(\xi)} \ {}_NY_{n,j}(\xi) \mathrm d\omega(\xi)\\
&= \frac{1}{4 \pi} \ \sum_{n=\vert N\vert}^L (2n+1) \int_R \mathrm d\omega(\xi)\\
&=\left((L+1)^2-N^2\right) \ \frac{A}{4\pi}. \qedhere
\end{align*}
\end{proof}

\noindent As an obvious consequence, the
number of Slepian functions with significant eigenvalues is higher
 for regions with a large area on the unit sphere, than it is for
regions covering a small area. For the special case of a
spherical cap ($b=\cos \theta \leq t \leq 1$), the area satisfies 
\begin{equation*} \frac{A_\mathrm{cap}}{4\pi}=\frac{1-b}{2}.\end{equation*}

\section{Scalar, Vector, and Tensor Slepian Functions}\label{ScVeTeSlep}

 The previously described construction of spin-weighted Slepian functions with the help of Theorem~\ref{relation_tens23} allows us to construct scalar, vector, and tensor Slepian functions. In particular for the tensor Slepian functions for spherical cap regions, this approach presents a previously unknown commuting operator approach.

\subsection{Scalar Slepian Functions}

The scalar Slepian functions have already been well investigated e.g.~in \cite{sneeuw, simons, simonspaper}. From the definition of the spin-weighted spherical harmonics and, consequently, from Theorem \ref{relation_tens23}, we know that the spin-weighted spherical harmonics of spin weight zero are the fully normalized spherical harmonics. We therefore obtain the scalar Slepian functions directly from the spin-weighted Slepian functions with spin weight zero.

\subsection{Vector Slepian Functions}

 We revisit the vector Slepian functions presented by Jahn and Bokor \cite{jahn2012} and Plattner and Simons \cite{plattner} by constructing them using the spin-weighted spherical-harmonic approach. A commuting operator using the classical approach was presented in \cite{jahn}. Here we build an alternative vector spherical-harmonic basis using the spin-weighted spherical harmonics.
\begin{defi}\label{swhbvf}
 The spin-weighted harmonic-based vector functions with bandlimit~$L$ are 
\begin{align*}
y_{n,j}^1(\xi)&:=y_{n,j}^{(1)}(\xi)=\xi Y_{n,j}(\xi),\\
y_{n,j}^2(\xi)&:=\frac{1}{\sqrt 2} \left(-y_{n,j}^{(2)}(\xi)+iy_{n,j}^{(3)}(\xi)\right)=\tau_+ \ {}_{+1}Y_{n,j}(\xi),\\
y_{n,j}^3(\xi)&:=-\frac{1}{\sqrt 2} \left(-y_{n,j}^{(2)}(\xi)-iy_{n,j}^{(3)}(\xi)\right)=\tau_- \ {}_{-1}Y_{n,j}(\xi),
\end{align*}
for $\xi\in\Omega$ and $n=0_i,\ldots,L$, where $-n\leq j\leq n$, with
\begin{equation*} 0_i=\begin{cases} 0& , i=1\\ 1& ,i=2,3 \end{cases}. \end{equation*}
\end{defi}

\begin{remark}\label{vecbasis}
The functions in Definition~\ref{swhbvf} form an orthonormal basis of $\mathrm{harm}_{0\dots L}(\Omega)$. Moreover,
 \begin{equation*}
 y_{n,j}^i(\xi) \cdot \overline{y_{m,l}^k(\xi)} = 0, \qquad \text{ if }i\neq k\text{, for all } \xi\in\Omega \text{ and all }n,j,m,l\quad \text{ (pointwise orthogonality)},
 \end{equation*}
where $\cdot$ denotes the standard inner product.
\end{remark}
\noindent Pointwise orthonormality follows from the pointwise orthonormality of $\xi, \tau_+$, and $\tau_-$. The spin-harmonic-based vector functions with maximum degree $L$ form a basis of~$\mathrm{harm}_{0\dots L}(\Omega)$ because of their non-degenerate linear relationship to the functions $y_{n,j}^{(i)}$ for $i\in \{1,2,3\}, n=0_i,\ldots,L$, and $-n\leq j\leq n$, which themselves form a basis of~$\mathrm{harm}_{0\dots L}(\Omega)$, see Definition~\ref{vectorharm}.
\newline
\newline
\noindent We can therefore represent any vector function~$\mathcal g \in \mathrm{harm}_{0\dots L}(\Omega)$ as a linear combination of the spin-weighted harmonic-based vector functions 
\begin{equation*} \mathcal g(\xi)=\sum_{i=1}^3 \sum_{n=0_i}^L \sum_{j=-n}^n g_{n,j}^i y_{n,j}^i(\xi)\label{slepfunc_vec}\end{equation*}
for all $\xi\in\Omega$ with coefficients 
\begin{equation*} g_{n,j}^i=\int_\Omega \mathcal g(\xi) \cdot \overline{y_{n,j}^i(\xi)} \ \mathrm d\omega(\xi)\end{equation*}
for all $(i,n,j)\in J_L$. Here, we used the set of indices
\begin{equation*} J_L:=\left\lbrace (i,n,j) \ \vert \ i=1,2,3; n=0_i,\dots, L; j=-n,\dots ,n\right\rbrace.\end{equation*}

\noindent The spatial concentration problem for bandlimited vector functions, independent of the selected basis, is 
\begin{equation}\label{veccon} \lambda = \frac{\int_R \mathcal g(\xi)\cdot \overline{\mathcal g(\xi)} \ \mathrm d \omega(\xi)}{\int_\Omega \mathcal g(\xi) \cdot \overline{\mathcal g(\xi)} \ \mathrm d\omega(\xi)}=\max.\end{equation}

\noindent The classical vector spherical harmonic functions $y^{(i)}_{n,j}$ for $(i,n,j)\in J_L$ lead to a blockdiagonal matrix, where the normal component is decoupled from the tangential component (see \cite{jahn2012,plattner}). As we show in the following, the pointwise orthogonality of the three types of spin-weighted harmonic-based functions $y^{1}_{n,j}$, $y^{2}_{n,j}$, and $y^{3}_{n,j}$ leads to a blockdiagonal matrix with three blocks. One for the radial component and two for the tangential component. Moreover, due to the decoupling, we can solve the concentration problem for each of the spin weights individually, allowing us to take full advantage of the derivations in previous sections. In particular, the commuting operator solution for spherical caps for spin-weights $0,\pm 1$ translates directly into vector Slepian functions.

\begin{prob}\label{probvec4}
Concentration problem~(\ref{veccon}) expressed in the spin-weighted harmonic-based vector functions from Definition~\ref{swhbvf} yields the eigenvalue problem
\begin{equation*} kg=\lambda g,\end{equation*}
where
\begin{equation*} k:=\begin{pmatrix}
k^1 & 0 & 0\\
0 & k^2 & 0\\
0 & 0 & k^3
\end{pmatrix}:=
\begin{pmatrix}
K^0 & 0 & 0\\
0 & K^{+1} & 0\\
0 & 0 & K^{-1}
\end{pmatrix} \in \mathbb R^{\left\lbrack 3(L+1)^2-2\right\rbrack \times \left\lbrack 3(L+1)^2-2\right\rbrack}\end{equation*}
with the matrices $k^i$ defined by their components
\begin{equation*} k_{nj,n'j'}^i:=\int_R \overline{y_{n,j}^i(\xi)} \cdot y_{n',j'}^i(\xi) \ \mathrm d\omega(\xi). \end{equation*}
\end{prob}

\noindent Hence, the vector problem reduces to three spin-weighted problems for spin weights $0$, $+1$, and $-1$, which we solved in the previous chapter. To find the eigenvectors $g_\alpha$ of Problem \ref{probvec4}, we simply pad the eigenvectors of the spin-weighted problems with zeros. Again, we sort the eigenvectors by decreasing eigenvalues. 
Note that this sorting implies that the types of vectors are not sorted any more. We represent this with a mapping $\tilde{\alpha}:\{1,\dots,3(L+1)^2-2\}\to\{1,\dots,(L+1)^2\}\times\{1,2,3\}$, $\alpha\mapsto(\tilde{\alpha}_1,\tilde{\alpha}_2)$ which associates the number $\alpha$ of a vectorial Slepian function to a vector type (i.e.\ here, a matrix block) $\tilde{\alpha}_2(\alpha)\in\{1,2,3\}$ and the number $\tilde{\alpha}_1(\alpha)$ of a spin-weighted scalar Slepian function. Correspondingly, we obtain the vector Slepian functions $\mathcal g_\alpha$ from the spin-weighted Slepian functions $\mathcal G_{\tilde{\alpha}_1(\alpha)}^{N}$ described in Equation~(\ref{spinslep}) by
\begin{equation*} \mathcal g_\alpha (\xi) =\begin{cases}
 \xi \mathcal G_{\tilde{\alpha}_1(\alpha)}^{0}(\xi) & \text{, if } \tilde{\alpha}_2(\alpha)=1,
 \\
 \tau_+ \mathcal G_{\tilde{\alpha}_1(\alpha)}^{+1}(\xi) & \text{, if } \tilde{\alpha}_2(\alpha)=2
 \\
\tau_- \mathcal G_{\tilde{\alpha}_1(\alpha)}^{-1}(\xi) & \text{, if } \tilde{\alpha}_2(\alpha)=3.
\end{cases}
\end{equation*}

for $\alpha=1,\dots,3(L+1)^2-2$.
\newline
\newline
\noindent We obtain the same Shannon numbers as \cite{plattner}, 
\begin{equation*} S_{\mathrm{vector}}= \left\lbrack 3(L+1)^2-2\right\rbrack \ \frac{A}{4\pi}. \end{equation*}
 For the special case of a spherical cap with $b=\cos \theta \leq t \leq 1$, this yields 
\begin{equation*} S_{\mathrm{vector}}= \left\lbrack 3(L+1)^2-2\right\rbrack \ \frac{1-b}{2}.\end{equation*}

\subsection{Tensor Slepian Functions}\label{tenslep}
 Tensor Slepian functions on the sphere have been investigated by Eshagh \cite{eshagh} with a choice of basis for which, to date, no commuting operator is known. Here, we follow the recipe used for the spin-weighted harmonic-based vector functions to construct tensor Slepian functions for which we derived a commuting operator for polar cap regions in Section~\ref{chapcomcap}. As for the vector case, the first step involves defining a basis of spin-weighted harmonic-based tensor functions. 
\begin{defi}\label{swhbtf}
 The spin-weighted harmonic-based tensor functions with bandlimit $L$ are 
\allowdisplaybreaks
\begin{align*}
\boldsymbol y_{n,j}^1(\xi):&=\boldsymbol y_{n,j}^{(1,1)}(\xi)=\left(\xi\otimes\xi\right) Y_{n,j}(\xi),\\
\boldsymbol y_{n,j}^2(\xi):&=\frac{1}{\sqrt 2}\left(-\boldsymbol y_{n,j}^{(1,2)}(\xi)+i\boldsymbol y_{n,j}^{(1,3)}(\xi)\right)=\left(\xi\otimes\tau_+\right) \ {}_{+1}Y_{n,j}(\xi),\\
\boldsymbol y_{n,j}^3(\xi):&=-\frac{1}{\sqrt 2}\left(-\boldsymbol y_{n,j}^{(1,2)}(\xi)-i\boldsymbol y_{n,j}^{(1,3)}(\xi)\right)=\left(\xi\otimes\tau_-\right) \ {}_{-1}Y_{n,j}(\xi),\\
\boldsymbol y_{n,j}^4(\xi):&=\frac{1}{\sqrt 2}\left(-\boldsymbol y_{n,j}^{(2,1)}(\xi)+i\boldsymbol y_{n,j}^{(3,1)}(\xi)\right)=\left(\tau_+\otimes\xi\right) \ {}_{+1}Y_{n,j}(\xi),\\
\boldsymbol y_{n,j}^5(\xi):&=-\frac{1}{\sqrt 2}\left(-\boldsymbol y_{n,j}^{(2,1)}(\xi)-i\boldsymbol y_{n,j}^{(3,1)}(\xi)\right)=\left(\tau_-\otimes\xi\right) \ {}_{-1}Y_{n,j}(\xi),\\
\boldsymbol y_{n,j}^6(\xi):&=\boldsymbol y_{n,j}^{(2,2)}(\xi)=\frac{1}{\sqrt{2}} \ \boldsymbol{\mathrm i}_{\mathrm{tan}} Y_{n,j}(\xi),\\
\boldsymbol y_{n,j}^7(\xi):&=\boldsymbol y_{n,j}^{(3,3)}(\xi)=\frac{1}{\sqrt{2}} \ \boldsymbol{\mathrm j}_{\mathrm{tan}} Y_{n,j}(\xi),\\
\boldsymbol y_{n,j}^8(\xi):&=-\frac{1}{\sqrt 2}\left(-\boldsymbol y_{n,j}^{(2,3)}(\xi)+i\boldsymbol y_{n,j}^{(3,2)}(\xi)\right)=\left(\tau_+\otimes\tau_+\right) \ {}_{+2}Y_{n,j}(\xi),\\
\boldsymbol y_{n,j}^9(\xi):&=-\frac{1}{\sqrt 2}\left(-\boldsymbol y_{n,j}^{(2,3)}(\xi)-i\boldsymbol y_{n,j}^{(3,2)}(\xi)\right)=\left(\tau_-\otimes\tau_-\right) \ {}_{-2}Y_{n,j}(\xi)
\end{align*}
\noindent for $\xi\in\Omega$ and $n=\boldsymbol 0_i,\ldots,L$, where $-n\leq j\leq n$, with 
\begin{equation*} \boldsymbol 0_i:=\begin{cases} 0& , i=1,6,7\\ 1& ,i=2,3,4,5\\ 2& ,i=8,9 \end{cases}.\end{equation*}
\end{defi}

\begin{remark}\label{tenbasis}
 
 The functions in Definition~\ref{swhbtf} form an orthonormal basis of $\boldsymbol{\mathrm{harm}}_{0\dots L}(\Omega)$. Moreover,
 \begin{equation*}
 \boldsymbol y_{n,j}^i(\xi) : \overline{\boldsymbol y_{m,l}^k(\xi)} = 0,
 \qquad \text{ if }i\neq k\text{ for all } \xi\in\Omega \text{ and all }n,j,m,l
 \end{equation*}
(pointwise orthogonality of different types) where $:$ denotes the tensor inner product.
 
\end{remark}

\noindent Remark~\ref{tenbasis} follows using the same arguments as for Remark~\ref{vecbasis}. 
\newline
\newline
 We can therefore represent any tensor function $ \mathbcal g \in \boldsymbol{\mathrm{harm}}_{0\dots L}(\Omega)$ as a linear combination of the spin-weighted harmonic-based tensor functions 
\begin{equation*} \mathbcal g(\xi)=\sum_{i=1}^9 \sum_{n=\boldsymbol 0_i}^L \sum_{j=-n}^n \boldsymbol g_{n,j}^i \boldsymbol y_{n,j}^i(\xi)\label{slepfunc_tens}\end{equation*}
for all $\xi\in\Omega$ with coefficients 
\begin{equation*} \boldsymbol g_{n,j}^i=\int_\Omega \mathbcal g(\xi) : \overline{\boldsymbol y_{n,j}^i(\xi)} \ \mathrm d\omega(\xi)\end{equation*}
for all $(i,n,j)\in\boldsymbol J_L$, where we define the set of indices by
\begin{equation*} \boldsymbol J_L:=\left\lbrace (i,n,j) \ \vert \ i=1,\dots ,9; n=\boldsymbol 0_i,\dots, L; j=-n,\dots ,n\right\rbrace\end{equation*}

\noindent We can formulate the concentration problem independently of the basis 
\begin{equation}\label{tencon} \lambda = \frac{\int_R\mathbcal g(\xi) : \overline{\mathbcal g(\xi)} \ \mathrm d \omega(\xi)}{\int_\Omega \mathbcal g(\xi) : \overline{\mathbcal g(\xi)} \ \mathrm d\omega(\xi)}=\max.\end{equation}

\noindent Similarly to the vector case, choosing the spin-weighted harmonic-based tensor basis from Definition~\ref{swhbtf} leads to a decoupling of the eigenvalue problem as a result of the pointwise orthogonality of the different types of spin-weighted harmonic-based tensor functions. 

\begin{prob}\label{probtens}
 Concentration problem~(\ref{tencon}) expressed in the spin-weighted harmonic-based tensor functions from Definition~\ref{swhbtf} yields the eigenvalue problem
\begin{equation*} \boldsymbol k\boldsymbol g=\lambda \boldsymbol g,\end{equation*}
\allowdisplaybreaks
where
\begin{align} \boldsymbol k&:=\begin{pmatrix}
\boldsymbol k^1 & 0 & 0 & 0 & 0 & 0 & 0 & 0 & 0 \\
0 & \boldsymbol k^2 & 0 & 0 & 0 & 0 & 0 & 0 & 0 \\
0 & 0 & \boldsymbol k^3 & 0 & 0 & 0 & 0 & 0 & 0 \\
0 & 0 & 0 & \boldsymbol k^4 & 0 & 0 & 0 & 0 & 0 \\
0 & 0 & 0 & 0 & \boldsymbol k^5 & 0 & 0 & 0 & 0 \\
0 & 0 & 0 & 0 & 0 & \boldsymbol k^6 & 0 & 0 & 0 \\
0 & 0 & 0 & 0 & 0 & 0 & \boldsymbol k^7 & 0 & 0 \\
0 & 0 & 0 & 0 & 0 & 0 & 0 & \boldsymbol k^8 & 0 \\
0 & 0 & 0 & 0 & 0 & 0 & 0 & 0 & \boldsymbol k^9
\end{pmatrix}\nonumber\\
&:=
\begin{pmatrix}
K^0 & 0 & 0 & 0 & 0 & 0 & 0 & 0 & 0\\
0 & K^{+1} & 0 & 0 & 0 & 0 & 0 & 0 & 0\\
0 & 0 & K^{-1} & 0 & 0 & 0 & 0 & 0 & 0\\
0 & 0 & 0 & K^{+1} & 0 & 0 & 0 & 0 & 0\\
0 & 0 & 0 & 0 & K^{-1} & 0 & 0 & 0 & 0\\
0 & 0 & 0 & 0 & 0 & K^0 & 0 & 0 & 0\\
0 & 0 & 0 & 0 & 0 & 0 & K^0 & 0 & 0\\
0 & 0 & 0 & 0 & 0 & 0 & 0 & K^{+2} & 0\\
0 & 0 & 0 & 0 & 0 & 0 & 0 & 0 & K^{-2}
\end{pmatrix}\in \mathbb R^{\left\lbrack 9(L+1)^2-12\right\rbrack \times \left\lbrack 9(L+1)^2-12\right\rbrack}\label{Eq:Matrblocktens}\end{align}
with the matrices $\boldsymbol k^i$ given by their components
\begin{equation*} \boldsymbol k_{nj,n'j'}^i:= \int_R \overline{\boldsymbol y_{n,j}^i(\xi)} : \boldsymbol y_{n',j'}^i(\xi) \ \mathrm d\omega(\xi). \end{equation*}
\end{prob}

\noindent The tensor problem reduces to nine spin-weighted problems corresponding to spin weights $0$, $+1$, $-1$, $+2$, and $-2$, which we solved in Section \ref{Sec:SPWSlepFct}. To find the eigenvectors $\boldsymbol g_\alpha$ of Problem~\ref{probtens}, we simply pad the eigenvectors of the spin-weighted problems with zeros. As is customary for Slepian functions, we sort the eigenvectors by decreasing eigenvalues and mix the different types. Similarly as in the vectorial case, we need again an index mapping $\tilde{\alpha}:\{1,\dots,9(L+1)^2-12\}\to\{1,\dots,(L+1)^2\}\times\{1,\dots,9\}$, where now $\tilde{\alpha}_2(\alpha)$ refers to the tensor type, i.e.\ the block of the matrix in \eqref{Eq:Matrblocktens}. We obtain the tensor Slepian functions $\mathbcal g_\alpha $ from the spin-weighted Slepian functions $\mathcal G_{\tilde{\alpha}_1(\alpha)}^N$ described in Equation~(\ref{spinslep}) by multiplying them with the corresponding unit tensor based on the tensor function type $i \in \{1,\ldots,9\}$. 
\begin{equation*}\label{SlepTensTypes} \mathbcal g_\alpha (\xi) =\begin{cases}
 \left(\xi\otimes\xi\right) \mathcal G^0_{\tilde{\alpha}_1(\alpha)} (\xi)& \text{, if }\tilde{\alpha}_2(\alpha)=1\\
 \left(\xi\otimes\tau_+\right) \mathcal G^{+1}_{\tilde{\alpha}_1(\alpha)} (\xi)& \text{, if }\tilde{\alpha}_2(\alpha)=2\\
 \left(\xi\otimes\tau_-\right) \mathcal G^{-1}_{\tilde{\alpha}_1(\alpha)} (\xi)& \text{, if }\tilde{\alpha}_2(\alpha)=3\\
 \left(\tau_+\otimes\xi\right) \mathcal G^{+1}_{\tilde{\alpha}_1(\alpha)} (\xi)& \text{, if }\tilde{\alpha}_2(\alpha)=4\\
 \left(\tau_-\otimes\xi\right) \mathcal G^{- 1}_{\tilde{\alpha}_1(\alpha)} (\xi)& \text{, if }\tilde{\alpha}_2(\alpha)=5\\
 \frac{1}{ \sqrt{2}} \ \boldsymbol{\mathrm i}_{\mathrm{tan}} \mathcal G^0_{\tilde{\alpha}_1(\alpha)} (\xi)& \text{, if }\tilde{\alpha}_2(\alpha)=6\\
 \frac{1}{ \sqrt{2}} \ \boldsymbol{\mathrm j}_{\mathrm{tan}} \mathcal G^0_{\tilde{\alpha}_1(\alpha)} (\xi)& \text{, if }\tilde{\alpha}_2(\alpha)=7\\
 \left(\tau_+\otimes\tau_+\right) \mathcal G^{+2}_{\tilde{\alpha}_1(\alpha)} (\xi)& \text{, if }\tilde{\alpha}_2(\alpha)=8\\
 \left(\tau_-\otimes\tau_-\right) \mathcal G^{-2}_{\tilde{\alpha}_1(\alpha)} (\xi)& \text{, if }\tilde{\alpha}_2(\alpha)=9
\end{cases}
\end{equation*}

\noindent We obtain the Shannon number
\begin{equation*} S_{\mathrm{tensor}}= \left\lbrack 9(L+1)^2-12 \right\rbrack \ \frac{A}{4\pi} \end{equation*}
 for general domains, and 
\begin{equation*} S_{\mathrm{tensor}}= \left\lbrack 9(L+1)^2-12 \right\rbrack \ \frac{1-b}{2},\end{equation*}
for the spherical cap with $b=\cos \theta \leq t \leq 1$.

\section{Conclusions}

Scalar and vector Slepian functions on the sphere have proven to
be a useful tool in a variety of studies. The construction of Slepian
functions can be rendered numerically stable and computationally
efficient through an aptly designed commuting operator. In this
article, we contributed to the understanding of Slepian functions in
two ways: (i) we presented a unified approach for constructing Slepian
functions for arbitrary rank tensors and (ii) we designed commuting
operators for polar caps for arbitrary rank tensor Slepian
functions. For the tensor Slepian functions, no such commuting
operator had been known.

\subsection{Summary of the construction} 

Designing the spin-weighted Slepian functions required us to
 solve the Slepian concentration problem for
 the spin-weighted spherical harmonics of general spin
weight $N$. We reformulated the concentration problem as a
 spin-weighted eigenvalue problem $K^NG^N=\lambda
G^N$. Furthermore, we derived a spin-weighted Shannon
number allowing for the estimation of the number of
eigenvalues close to~$1$ and hence the number of Slepian
functions that are well concentrated within the region of
interest.
\newline
\newline
 For the spin-weighted kernel function $\mathcal K^N$ for polar
cap regions we derived a commuting operator $\mathcal I^N$ 
leading to a tridiagonal matrix $I^N$ which commutes with the kernel
matrix $K^N$ and which has simple eigenvalues. As a result, the
eigenvectors of $I^N$ (which are numerically stable and
computationally inexpensive to compute) are equal to the eigenvectors
of $K^N$ (which are the coefficients for the spin-weighted
spherical-harmonic-based Slepian functions).
\newline
\newline
We used the linear relationships between the
spin-weighted spherical harmonics and the scalar
spherical harmonics, the vector spherical harmonics by Hill
\cite{hill}, and the tensor spherical harmonics by Freeden, Gervens,
and Schreiner \cite{freedenpaper} to design the corresponding Slepian
functions on the sphere. For the tensor Slepian functions, this work
presents the first construction of a commuting operator. 

\subsection{Outlook}

We presented a framework for the construction of tensor Slepian
functions for the spherical cap as well as for 
arbitrary measurable domains but not implemented it. 
Development of
software to construct tensor Slepian functions for arbitrary domains
hence remains an open problem. If such a software were to be designed
to solve the concentration problem for general spin-weighted
spherical-harmonics, then this would allow the construction of tensor
Slepian functions for arbitrary ranks and arbitrary regions. 
Moreover, the construction of a commuting operator for the polar
 double cap and belt for the vector and tensor Slepian functions
is unknown at present. While the Shannon number does provide an
estimation for the number of well-concentrated Slepian functions, it
typically overestimates that number. Hence a better constraint on the
number of well-concentrated Slepian functions would provide a valuable
contribution. As another potential avenue for future research, tensor
Slepian functions could form part of a dictionary-based method for
tensor-valued inverse problems in analogy to \cite{FischerMichel2012,MichelTelschow2016,roger}.
\newline
\newline
 To date, tensor Slepian functions have not been used to invert
for potential field models from second-derivative data such as, for
example, for gravity potential from the satellite mission GOCE,
or to invert for cosmic microwave background polarization. As
is shown in \cite{plattnerAltCog}, when inverting for
potential field models on the planet's surface from satellite data,
the spatially concentrated spectrally limited Slepian functions are
not well suited, as they are typically poorly conditioned under
downward continuation (as a result of each function including a wide
range of spherical-harmonic degrees). The approach of solving a
related, continuation-cognizant problem as in \cite{plattnerAltCog}
could be translated to the tensor Slepian case.

\bibliography{Literatur}
\bibliographystyle{abbrv}

\end{document}